\documentclass{amsart}
\usepackage{tikz}
\usepackage{xcolor}
\usepackage{amssymb,latexsym,amsmath,extarrows, mathabx}
\usepackage{graphicx,mathrsfs,comment}
\usepackage{hyperref,url}
\usepackage{pict2e}
\usepackage{enumerate}

\usepackage{cancel}

\usepackage{amstext}
\usepackage{bbm} 

\numberwithin{equation}{section}

\usepackage{esint}

\newtheorem*{acknowledgement}{Acknowledgements}
\newtheorem{theorem}{Theorem}[section]
\newtheorem{lemma}[theorem]{Lemma}

\newtheorem{proposition}[theorem]{Proposition}
\newtheorem{remark}[theorem]{Remark}
\newtheorem*{remark*}{Remark}

\newtheorem{definition}[theorem]{Definition}

\newtheorem{conjecture}[theorem]{Conjecture}

\newcommand{\C}{\mathbb{C}}

\newcommand{\cP}{\mathcal P}

\makeatletter
\newcommand{\barredsum}{%
  \DOTSB\mathop{\mathpalette\@barredsum\relax}\slimits@
}
\newcommand{\@barredsum}[2]{%
  \begingroup
  \sbox\z@{$#1\sum$}%
  \setlength{\unitlength}{\dimexpr2pt+\ht\z@+\dp\z@\relax}%
  \@barredsumthickness{#1}%
  \vphantom{\@barredsumbar}%
  \ooalign{$\m@th#1\sum$\cr\hidewidth$#1\@barredsumbar$\hidewidth\cr}%
  \endgroup
}
\newcommand{\@barredsumbar}{%
  \vcenter{\hbox{\begin{picture}(0,1)\roundcap\Line(0,0)(0,1)\end{picture}}}%
}
\newcommand{\@barredsumthickness}[1]{
  \linethickness{%
    1.25\fontdimen8
      \ifx#1\displaystyle\textfont\else
      \ifx#1\textstyle\textfont\else
      \ifx#1\scriptstyle\scriptfont\else
      \scriptscriptfont\fi\fi\fi 3
  }%
}
\makeatother

\newcommand{\be}{\beta}
\newcommand{\ga}{\gamma}

\newcommand{\de}{\delta}

\newcommand{\e}{\varepsilon}

\newcommand{\la}{\lambda}

\newcommand{\si}{\sigma}

\newcommand{\vp}{\varphi}
\newcommand{\om}{\omega}

\newcommand{\cs}{\mathcal S}
\newcommand{\cv}{\mathcal V}

\newcommand{\cj}{\mathcal J}

\newcommand{\wh}{\widehat}

\newcommand{\ZR}{\mathbb{R}}
\newcommand{\ZT}{\mathbb{T}}
\newcommand{\ZH}{\mathbb{H}}

\newcommand{\ZS}{\mathbb{S}}

\newcommand{\Id}{{\bf{1}}}

\newcommand{\Tau}{\mathcal{T}}

\newcommand{\cT}{{\mathcal T}}

\newcommand{\cC}{{\mathcal C}}

\newcommand{\R}{\mathbb{R}}

\newcommand{\Br}{{\rm{Br}}}

\newcommand{\dist}{{\rm dist}}

\newcommand{\supp}{{\rm supp}}

\begin{document}

\title[Restriction and decoupling estimates]{Restriction and decoupling estimates for the hyperbolic paraboloid in $\R^3$}

\author{Ciprian Demeter} \address{ Ciprian Demeter\\  Department of Mathematics\\ Indiana University Bloomington, USA} \email{demeterc@iu.edu}

\author{Shukun Wu} \address{ Shukun Wu\\  Department of Mathematics\\ Indiana University Bloomington, USA} \email{shukwu@iu.edu}

\thanks{The first author is partially supported by the NSF grants DMS-2055156 and  DMS-2349828.}

\thanks{The second author is partially supported by the NSF grant NSF-2453583.}

\begin{abstract}
We prove bilinear $\ell^2$-decoupling and refined bilinear decoupling inequalities for the  truncated hyperbolic paraboloid in $\R^3$.
As an application, we prove the associated restriction estimate in the range $p>22/7$, matching an earlier result for the elliptic paraboloid.

\end{abstract}

\maketitle

\section{Introduction}

\subsection{Overview} Let $S\subset\ZR^n$ be a smooth compact hypersurface and let $d\si_S$ be its surface measure.
We consider the associated extension operator 
\begin{equation}
\nonumber
    \wh{fd\si_S}(x)=\int e^{2\pi ix\cdot \xi}f(\xi)d\si_S(\xi).
\end{equation}
Elias Stein conjectured the following.

\begin{conjecture}
\label{restriction-conj}
When $S$ has non-vanishing Gaussian curvature, 
\begin{equation}
\label{restriction-esti-1}
    \|\wh{fd\si_S}\|_{p}\lesssim \|f\|_{L^p(d\si_S)}
\end{equation}
holds for all $p>\frac{2n}{n-1}$ and all smooth functions $f$ on $S$.
\end{conjecture}

Since Bourgain's work \cite{Bourgain-Besicovitch}, Conjecture \ref{restriction-conj} has been studied intensively.
Most recently, \cite{Wang-Wu} posted an incidence geometry conjecture that, along with decoupling theorems, would fully solve Conjecture \ref{restriction-conj} when $S$ is of elliptic type (though this was later invalidated by Cohen's example \cite{Cohen-example}).
One notable property of elliptic surfaces is that they do not contain any linear subspaces, which are typical sources of constructive interference.
For example, the $\ell^2$-decoupling theorem in \cite{Bourgain-Demeter-l2} is known to fail when $S$ is not of elliptic type. 
See \cite{BDh}. 

However, the existence of linear subspaces does not invalidate the $L^p$-estimate \eqref{restriction-esti-1}, since a surface with non-vanishing Gaussian curvature cannot contain linear subspaces of large dimension.
Moreover, it is conceivable that linear subspaces are the only obstruction to orthogonality results such as the decoupling theorem.
In other words, if constructive interference from linear subspaces is neutralized, then an appropriate form of  $\ell^2$-decoupling  may still hold.

\smallskip

In this paper we prove decoupling inequalities that support the aforementioned philosophy.
Specifically, in the setting governed by transversality (Definition \ref{transversality}), we establish both a bilinear $\ell^2$-decoupling  and a bilinear refined decoupling inequality for functions whose Fourier transforms are supported near the hyperbolic paraboloid
\begin{equation}
\label{hyperboloid}
    \ZH=\{(\xi,\eta,\xi\eta):\;(\xi,\eta)\in\R^2\}.
\end{equation}
As an application, we prove \eqref{restriction-esti-1} for $p>22/7$ when $S=\ZH\cap [-1,1]^3$, matching the best-known result in \cite{Wang-Wu} for elliptic surfaces in $\ZR^3$.

Let us now briefly describe our results.

\medskip

\subsection{Bilinear decoupling inequalities}

We first introduce some notation and a critical definition.

\smallskip

For a rectangle $\tau\subset [-1,1]^2$, define $\ZH_\tau=\{(\xi,\eta,\xi\eta):(\xi,\eta)\in\tau\}$.
For a function of three variables $f:\R^3\to\C$, we write $f_\tau$  for the Fourier restriction of $f$ to $\tau\times \R$.
We write $\cP_\Delta(\tau)$ for the collection of $\Delta$-squares in some partition of $\tau$. 
\begin{definition}
\label{transversality}
	We call a pair of squares (of arbitrary size) $\tau_1,\tau_2\subset [-1,1]^2$ {\bf transverse} if $\dist(\xi_1,\xi_2)\sim 1$ and $\dist(\eta_1,\eta_2)\sim 1$ for each $(\xi_j,\eta_j)\in \tau_j$.
\end{definition}    
Transversality is in fact equivalent to asking that $\dist(\tau_1,\tau_2)\sim 1$ and also that each line joining some $(\xi_1,\eta_1)\in\tau_1$ and $(\xi_2,\eta_2)\in\tau_2$ has slope with absolute value satisfying $$\left|\frac{\eta_2-\eta_1}{\xi_2-\xi_1}\right|\sim 1.$$
In particular,  the line 
    $\ell(\tau_1,\tau_2)$ joining the centers of such squares
is (quantitatively) transverse to both coordinate axes.	All lines contained in $\ZH$ are parallel to either the plane $\xi=0$ or the plane $\eta=0$. Transversality guarantees that none of these lines intersects both $\ZH_{\tau_1}$ and $\ZH_{\tau_2}$.

\medskip

\subsubsection{$\ell^2$-decoupling}

\begin{definition}[Bilinear decoupling constant for $\ell^2$-decoupling]
Given $0<\delta<1$ and $R\ge 1$, we let $C(\delta,R)$ be the smallest constant such that 
\begin{equation}
\label{l2-decoupling}
    \int_{\R^3}|f_1f_2|^2\le C(\delta,R) \prod_{j=1}^2\Big(\sum_{\theta_j\in \cP_{R^{-1/2}}(\tau_j)}\big\|f_{\theta_j}\big\|_{L^4(\R^3)}^2\Big)
\end{equation}
for each transverse $\delta$-squares $\tau_1,\tau_2$ and  each  $f_j$ Fourier supported on the $1/R$-neighborhood $N_{1/R}(\ZH_{\tau_j})$ of $\ZH_{\tau_j}$.    
\end{definition}

\begin{remark}
\rm

It is clear that $C(\delta,R)$ is nondecreasing in $\delta$, and, at least heuristically, it is also nondecreasing in $R$.

\end{remark}

\begin{remark}
\rm

Due to the Fourier support of $f_1$ and $f_2$, \eqref{l2-decoupling} implies (in fact, it is equivalent to) a localized version of itself, with $\R^3$ replaced on both sides by (some smooth approximations of) $1_{B_R}$.
\end{remark}
Here is our first result.
\begin{theorem}[Bilinear $\ell^2$ decoupling]
\label{dec-thm}
For all $\e>0$, we have  $C(1,R)\lesssim_\e R^{\e}$.
\end{theorem}

\smallskip

The proof of Theorem \ref{dec-thm} is inspired by the alternative proof of the elliptic $\ell^2$ decoupling theorem given in \cite{FGWW}.
The key new observation in the non-elliptic setting here is the following. Let $\tau_1$ and $\tau_2$ be two transverse $\de$-squares in $[-1,1]^2$, and for $j=1,2$, let $S_{\tau_j}=N_{\de^2}(H_{\tau_j})$ be an approximate $\de\times\de\times\de^2$-box.
Then, interpreting $S_{\tau_1}$ and $S_{\tau_2}$ as $\de^2$-neighborhoods of two planes $\pi_1,\pi_2$, the intersection $\pi_1\cap\pi_2$ is a line whose  projection   onto the horizontal $(\xi,\eta)$-plane  is transverse to both coordinate directions.
We refer the reader to the proof of Proposition \ref{jfiriurgutgur8g} for details.

At the foundation of all our orthogonality arguments lies the following classical equivalence.
    \begin{proposition}(Bilinear restriction)
    \label{labell}
    Let $\gamma_1,\gamma_2$ be two smooth curves in $\R^2$, such that any two of their normal vectors $n_1,n_2$ are (quantitatively) transverse. Assume $f_j$ is Fourier supported in the $\Delta$-neighborhood $N_\Delta(\gamma_j)$ of $\gamma_j$. Partition $N_\Delta(\gamma_j)$ into $\Delta$-squares $s_j$. Then
    \begin{equation}
\label{firegitguihui09yhui09y}    
    \int_{\R^2}|f_1f_2|^2\sim \sum_{s_1,s_2}\int_{\R^2}|P_{s_1}f_1P_{s_2}f_2|^2,
    \end{equation}
    where $P_sf$ is the Fourier restriction of $f$ to $s$.\end{proposition}
    This equivalence may be easily proved using simple geometric arguments that rely critically on  the fact that $4=2\times 2$. This type of argument is sometimes referred to as {\em bi-orthogonality}. However, the paper \cite{BCT}
    revealed that \eqref{firegitguihui09yhui09y} is the two-dimensional manifestation of the more general multilinear restriction phenomenon in $\R^n$, that registers at the critical exponent $\frac{2n}{n-1}$. 
    With this perspective came a different proof of \eqref{firegitguihui09yhui09y},  that presents a severe departure from bi-orthogonality. 
    Our proof of Theorem \ref{dec-thm} embraces this philosophy, leading to a bi-orthogonality free argument for the bilinear decoupling inequality for the two-dimensional paraboloid (both elliptic and hyperbolic). 
    In the elliptic case, the standard bilinear-to-linear reduction immediately recovers the linear $\ell^2$ decoupling proved in \cite{Bourgain-Demeter-l2}, without the use of the trilinear restriction theorem from \cite{BCT}. 

    It remains an interesting open problem to extend our results to higher dimensions.  Our bi-orthogonality free argument opens up the possibility for a similar argument in $\R^n$, when $n$ is odd. By this we mean, a proof of $d$-linear $\ell^2$ decoupling in $\R^n$ using the $d$-linear restriction theorem in $\R^d$. 
    This  speculation is entertained by the coincidence between the multilinear restriction exponent $\frac{2d}{d-1}$ in $\R^d$ and the critical exponent $\frac{2(n+1)}{n-1}$ for $\ell^2$ decoupling in $\R^n$, when $n=2d-1$. However, while this numerology is consistent in critical places of the argument,  there are new difficulties in higher dimensions. These are associated with the more complex broad-narrow reduction, when trying to establish the analog of inequality
\eqref{igjg8u8956ug896u86u8}. We mention that the coincidence between the two exponents was recently exploited in \cite{ChOh}, in order to produce a proof of $\ell^p$ (rather than $\ell^2$) decoupling, albeit conditional to the Restriction Conjecture.

\begin{remark}
\rm

By a standard broad-narrow argument, Theorem \ref{dec-thm}  recovers the decoupling inequality for $\ZH$ (Theorem 1.8) from the recent paper \cite{GMO}.
See Section \eqref{app-linear-dec-sec}.
\end{remark}

\medskip

\subsubsection{Refined decoupling}

The statement of the refined decoupling inequality relies on the wave packet decomposition for functions supported on a thin neighborhood of a surface $S$.
We refer the reader to subsection \ref{ssWP} for notation and the details of this decomposition. 

\begin{definition}[Decoupling constant for bilinear refined decoupling]

Let $\tau_1,\tau_2\subset[-1,1]^2$ be two transverse squares (of arbitrary size), and let $X$ be the union of a collection of pairwise disjoint $R^{1/2}$-balls $Q$ inside $B_R$. 
For $j=1,2$, let $f_j=\sum_{T\in\ZT_j}f_T$ be a sum of scale $R$ wave packets so that ${\rm supp}(\wh f_j)\subset N_{R^{-1}}(\ZH_{\tau_j})$.
Suppose that  there is $M_j\geq1$ such that each $R^{1/2}$-ball $Q\subset X$ intersects at most $M_j$ many $R$-tubes from $\ZT_j$.

We define $C(R)$ to be the smallest constant such that for all such configurations, the following inequality  holds:
\begin{equation}
\label{ref-dec-esti}	
    \int_X|f_1f_2|^2\le C(R) (M_1M_2)^{1/2}\prod_{j=1}^2\Big(\sum_{T\in\ZT_j}\big\|f_{T_j}\big\|_4^4\Big)^{1/2}.   
\end{equation}	   
\end{definition}
We prove the following result. 
\begin{theorem}(Bilinear refined decoupling)
\label{refined-dec-thm}
For all $\e>0$, we have	$C(R)\lesssim_\e R^\e$.
\end{theorem}

\smallskip

The (linear) refined decoupling inequality for elliptic surfaces was introduced in \cite{GIOW}.
Its proof relied critically on the (linear) $\ell^2$-decoupling from \cite{Bourgain-Demeter-l2}.
More precisely, this $\ell^2$-decoupling was applied on the smaller balls $Q$, leading to an elegant and easy-to-iterate inequality for the (linear) refined decoupling constant $C_{lin}(R)$, of the form 
$$C_{lin}(R)\lesssim_\e R^{\e}C_{lin}(\sqrt{R}).$$
However, this approach fails rather dramatically in the non-elliptic case of our Theorem \ref{refined-dec-thm}, due to the inefficiency of rescaling in the bilinear setting.

While our proof of Theorem \ref{refined-dec-thm} borrows some inspiration from the argument in Theorem \ref{dec-thm}, it needs a few new significant layers that essentially add up to new methodology. 
One of its main innovations is a multi-scale decomposition that preserves the bilinear structure at every scale. We employ a careful selection of the scale increment, that is consistent with unambiguous orientation for the emerging rectangles.
Perhaps somewhat counter intuitively, we iterate decoupling on small balls $Q$ of radius $\sqrt{R}$, rather than on $B_R$. We introduce a stopping time $K_3^{-1}$ for the frequency scale. There are two possibilities for decoupling to come to a halt. 
One is that we cross the threshold $K_3^{-1}$ while maintaining the critical bilinear structure, which is amenable to rescaling. The other one is that we decouple all the way down to the smallest scale $R^{-1/2}$. In this latter case, by the uncertainty principle, we are in fact proving a very satisfactory reverse square function estimate.

\medskip

\subsection{Restriction estimates}

As an application of our bilinear refined decoupling theorem, we prove the following restriction estimate.
\begin{theorem}
\label{restriction-thm}
When $n=3$, the restriction estimate \eqref{restriction-esti-1} is true when $S$ is the truncated hyperbolic paraboloid $\ZH_{[-1,1]^2}
$ and $p>22/7$. 
\end{theorem}

The previous best-known result is due to \cite{Cho-Lee}, where the authors use polynomial partitioning to prove \eqref{restriction-esti-1} for $p>3.25$.
For a generalization of this result to compact surfaces in $\ZR^3$ with non-zero Gaussian curvature, see \cite{Guo-Oh}.

\smallskip

The proof of Theorem \ref{restriction-thm} uses incidence estimates established in \cite{Wang-Wu}.
We remark that when it comes to our use of incidence geometry, there is no difference between the hyperbolic and the elliptic paraboloid. This is because in both cases, the normal vector is both injective and (essentially) surjective.  

\bigskip 

\begin{acknowledgement}
\rm
		We are grateful to Jacob Glidewell and Shengwen Gan for pointing out typos in the earlier version of the manuscript.
	\end{acknowledgement}
    
\bigskip

\noindent {\bf Notation:} 
Throughout the paper, we use $\# E$ to denote the cardinality of a finite set.
For $A,B\geq 0$, we use $A\lesssim B$ to mean $A\leq CB$ for an absolute constant (independent of scales) $C$, and use $A\sim B$ to mean $A\lesssim B$ and $B\lesssim A$.
For a given $\de<1$, we use $ A \lessapprox B$ to denote $A\leq c_{\upsilon}\de^{-\upsilon} B$ for all $\upsilon>0$ (same notation applies to a given $R>1$ by taking $\de=R^{-1}$). 
We use $B_R$ to denote a ball of radius $R$ in $\ZR^3$.

\bigskip

\section{The bilinear \texorpdfstring{$\ell^2$}{}-decoupling inequality}

	We will prove Theorem \ref{dec-thm} using induction on both $\delta$ and $R$.
	Note that if $\tau_1,\tau_2$ are transverse and $\tau_j'\subset \tau_{j}$, then $\tau_1'$ and $\tau_2'$ are also transverse. 
    Thus, a simple application of the triangle inequality (cover $\delta K$-caps by $\delta$-caps) shows that if $K\ge 1$
	\begin{equation}
		\label{hguugyrtugyuy5yty}
		C(\delta K,R)\lesssim K^{O(1)}C(\delta,R).
	\end{equation}    
    Here is our chief analytic tool.
	\begin{lemma}
		\label{main-tool}
		Consider two planes $\pi_1$, $\pi_2$, whose angle is $\sim 1$. Let $\ell$ be their common line. Assume $F_j$ is Fourier supported on the $\Delta$-neighborhood $N_\Delta(\pi_j)$ of $\pi_j$. Partition $N_\Delta(\pi_j)$ into rectangular boxes $b_j$ congruent to $[-\Delta,\Delta]\times [-\Delta,\Delta]\times \R$, whose infinite axis is parallel to $\ell$. 
	Then 
		$$\int_{\R^3}|F_1F_2|^2\sim \sum_{b_1,b_2}\int_{\R^3}|P_{b_1}F_1P_{b_2}F_2|^2,$$
		where $P_b F$ is the Fourier restriction of $F$ to $b$.
	\end{lemma}
	\begin{proof}
		Use Proposition \ref{labell} in a plane orthogonal to $\ell$, whose intersections with $\pi_1,\pi_2$ are transverse lines. Extend the equivalence to the planes via Fubini. \qedhere
		
	\end{proof}
	
	Here is the base case of the induction.
	
	\begin{proposition}
		\label{jfiriurgutgur8g}	
		$$C(R^{-1/4},R)\lesssim 1.$$
	\end{proposition}
	\begin{proof}
		Fix transverse $R^{-1/4}$-squares $\tau_1$, $\tau_2$, centered at $(\xi_1^*,\eta_1^*)$, $(\xi_2^*,\eta_2^*)$.
		Fix two functions $f_j$ Fourier supported on $N_{R^{-1}}(\ZH_{\tau_j})$.
		
		Consider the $R^{-1/2}$-neighborhood of $\ZH_{\tau_j}$. It lies inside the $O(R^{-1/2})$-neighborhood of the tangent plane at any $(\xi_j,\eta_j,\xi_j\eta_j)\in\ZH_{\tau_j}$, whose normal is 
		$$n(\xi_j,\eta_j)=(\eta_j,\xi_j,-1).$$
		The intersection of these planes is the line $\ell(\xi_1,\eta_1,\xi_2,\eta_2)$
		with direction
		$$(\xi_2-\xi_1,\eta_1-\eta_2,\eta_1\xi_2-\eta_2\xi_1).$$
		
        \medskip
		
	\noindent{\bf Step 1.} 
		We apply Lemma \ref{main-tool} with $\Delta=R^{-1/2}$ and get
		\begin{equation}
			\label{jjurgytg8yu869}	
			\int_{\R^3}|f_1f_2|^2\lesssim \sum_{\omega_1,\omega_2}\int_{\R^3}|f_{\omega_1}f_{\omega_2}|^2.
		\end{equation}
		Here, $\omega_j$ are rectangles that partition (or rather cover; this distinction will be ignored) $\tau_j$, with dimensions $\sim (R^{-1/2},R^{-1/4})$ and  long side in the direction $(\xi_2^*-\xi_1^*,\eta_1^*-\eta_2^*)$. Let us explain why there is no ambiguity with this choice. The orientation of a rectangle is only defined within an error comparable to its eccentricity $R^{-1/4}$.
		For any other choice of $(\xi_j',\eta_j')\in \tau_j$, the angle between directions $(\xi_2-\xi_1,\eta_1-\eta_2)$ and $(\xi_2'-\xi_1',\eta_1'-\eta_2')$ can be easily seen to be $\lesssim R^{-1/4}$. Thus,  $\omega_j$ are essentially uniquely determined.  With our concrete choice for the direction $(\xi_2^*-\xi_1^*,\eta_1^*-\eta_2^*)$, the sets $\omega_j$ are fully determined.
		
		For future reference, we note that the slope $\frac{\eta_1^*-\eta_2^*}{\xi_2^*-\xi_1^*}$ of this direction equals minus the slope of $\ell(\tau_1,\tau_2).$ 	

        \bigskip 
	
	\noindent	{\bf Step 2.} We examine each $f_{\omega_j}$.  Its Fourier support lies inside $N_{1/R}(\ZH_{\omega_j})$. The long side of $\omega_j$ points in a direction with slope of absolute value $\sim 1$. The part of $\ZH$ lying above any such line is a parabola with curvature $\sim 1$. The whole $\ZH_{\omega_j}$ is then within the $O(R^{-1})$-neighborhood of a parabolic cylinder with ``height" $\sim R^{-1/2}$, over an arc of the parabola of length $\sim R^{-1/4}$. Thus, the Fourier support of $f_{\omega_j}$ lies inside a similar neighborhood, as $O(R^{-1/2})+O(1/R)=O(R^{-1/2})$.
		
		We may use cylindrical $\ell^2(L^4)$ decoupling (planar decoupling for the arc of the parabola combined with Fubini in the ``height" direction)  to find
		$$\big(\int_{\R^3}|f_{\omega_j}|^4\big)^{1/2}\lesssim \sum _{\theta_j\in \cP_{R^{-1/2}}(\omega_j)}\|f_{\theta_j}\|^2_{L^4(\R^3)}.$$ 
		Combining this with \eqref{jjurgytg8yu869} and H\"older's inequality delivers the conclusion.\qedhere

	\end{proof}	 
	We note that $\delta\sim R^{-1/4}$ is the largest $\delta$ for which we get the desired decoupling directly. Larger values will require induction on scales.  We fix the parameter $K$, that will later be chosen to be $\lessapprox 1$.
    \begin{proposition}
		\label{diowue8tu954iy967u989}
		If $R\ge K\delta^{-3}$ we have (for some universal $C_1$, $C_2$, independent of $K,\delta,R$)
		\begin{equation}
			\label{bfjrhfurhgrugy}
			C(\delta,R)\le C_2( C(\delta/K,R)+\sup_{R'\le R\delta^2}K^{C_1}C(\delta,R')).
		\end{equation}
	\end{proposition}
	\begin{proof}
		Fix transverse $\delta$-squares $\tau_1$, $\tau_2$ centered at $(\xi_1^*,\eta_1^*)$, $(\xi_2^*,\eta_2^*)$, and two functions $f_j$ Fourier supported on $N_{R^{-1}}(\ZH_{\tau_j})$.
	\medskip
		
	\noindent	{\bf Step 1.} By repeating the argument from Step 1 in the proof of Proposition \ref{jfiriurgutgur8g}, we have 
		\begin{equation}
			\label{jjurgytg8yu869i09i09hhi90h}	
			\int_{\R^3}|f_1f_2|^2\lesssim \sum_{\omega_1,\omega_2}\int_{\R^3}|f_{\omega_1}f_{\omega_2}|^2.
		\end{equation}
		Here $\omega_j$ are thin rectangles with dimensions $\sim (\delta^2,\delta)$, pointing in the direction $(\xi_2^*-\xi_1^*,\eta_1^*-\eta_2^*)$ .

    \bigskip
		
	\noindent	{\bf Step 2.} We divide each $\omega_j$ into $K$ rectangles $s_j$ with dimensions $\sim (\delta^2,\delta/K)$, so
		$$f_{\omega_j}=\sum_{s_j\subset \omega_j}f_{s_j}.$$
		We write $s_j\not\sim s_j'$ if $s_j$ is not adjacent to $s_j'$.
	For each $x\in\R^3$,
		$$|f_{\omega_j}(x)|\le 10(\sum_{s_j\subset \omega_j}|f_{s_j}(x)|^2)^{1/2}+K^2\max_{s_j\not\sim s_{j}'\subset \omega_j}|f_{s_j}(x)f_{s_j'}(x)|^{1/2}.$$
		We call the first expression $S_{\omega_j}f(x)$.
		
		Let $B$ be a ball of radius $K/\delta$. Partition $B$ into sets $B_n$ and  $B_b$ as follows. We put $x$ in $B_n$ if $|f_{\omega_j}(x)|\lesssim S_{\omega_j}f(x)$ for at least one $j\in\{1,2\}$. It follows that
		$$\int_{B_n}|f_{\omega_1}f_{\omega_2}|^2\lesssim \int_{B}(S_{\omega_1}f)^2|f_{\omega_2}|^2+\int_{B}(S_{\omega_2}f)^2|f_{\omega_1}|^2.$$
		Since each $s_j$ lies inside a disk of radius $\sim \delta/K$,  the uncertainty principle shows that $|f_{s_j}(x)|$, and thus also each $S_{\omega_j}f(x)$, is essentially constant on $B$. Call $S_{\omega_j}f(B)$ the value of this constant. 
		It follows that for $j\not=j'\in\{1,2\}$,
		$$\int_{B}(S_{\omega_j}f)^2|f_{\omega_{j'}}|^2\approx (S_{\omega_j}f(B))^2\int_B|f_{\omega_{j'}}|^2.$$
		Furthermore, due to $L^2$ orthogonality we have
		$$\int_B|f_{\omega_{j'}}|^2\lesssim \int_B (S_{\omega_{j'}}f)^2.$$
		We conclude that
		$$\int_{B_n}|f_{\omega_1}f_{\omega_2}|^2\lesssim \int_{B}(S_{\omega_1}f)^2(S_{\omega_2f})^2.$$
		Also,
		$$\int_{B_b}|f_{\omega_1}f_{\omega_2}|^2\lesssim K^{O(1)}\max_{s_1\not\sim s_{1}'\subset \omega_1}\max_{s_2\not\sim s_{2}'\subset \omega_2}\int_B |f_{s_1}f_{s_1'}f_{s_2}f_{s_2'}|.$$We conclude this step by summing the last two inequalities over a finitely overlapping cover of $\R^3$ by balls $B$
		$$\int_{\R^3}|f_{\omega_1}f_{\omega_2}|^2\lesssim$$
		\begin{equation}
			\label{igjg8u8956ug896u86u8}	
			\int_{\R^3}(S_{\omega_1}f)^2(S_{\omega_2f})^2+K^{O(1)}\max_{s_1\not\sim s_{1}'\subset \omega_1}\max_{s_2\not\sim s_{2}'\subset \omega_2}\int_{\R^3} |f_{s_1}f_{s_1'}f_{s_2}f_{s_2'}|.
		\end{equation}
        Before we move on, let us note that our proof of \eqref{igjg8u8956ug896u86u8}	did not use geometric arguments specific to the use of $L^4$.  
	\medskip
		
	\noindent	{\bf Step 3.} We analyze the first term in \eqref{igjg8u8956ug896u86u8}.
		$$\int_{\R^3}(S_{\omega_1}f)^2(S_{\omega_2}f)^2=\sum_{s_j\subset \omega_j}\int_{\R^3}|f_{s_1}f_{s_2}|^2.$$
		Since $s_1,s_2$ lie inside transverse  $\delta/K$-squares, 
		$$\int_{\R^3}|f_{s_1}f_{s_2}|^2\le C(\delta/K,R)\prod_{j=1}^2\Big(\sum_{\theta_j\in \cP_{R^{-1/2}}(s_j)}\big\|f_{\theta_j}\big\|_{L^4(\R^3)}^2\Big).$$
		Thus,
		$$\int_{\R^3}(S_{\omega_1}f)^2(S_{\omega_2f})^2\le C(\delta/K,R)\prod_{j=1}^2\Big(\sum_{\theta_j\in \cP_{R^{-1/2}}(\omega_j)}\big\|f_{\theta_j}\big\|_{L^4(\R^3)}^2\Big).$$
		Summing in $\omega_1,\omega_2$ leads to the first upper bound in \eqref{bfjrhfurhgrugy}.
	\bigskip
		
	\noindent	{\bf Step 4.} We analyze the second term in \eqref{igjg8u8956ug896u86u8}. First, by H\"older's inequality,
		$$\int_{\R^3} |f_{s_1}f_{s_1'}f_{s_2}f_{s_2'}|\le (\int_{\R^3} |f_{s_1}f_{s_1'}|^2)^{1/2}(\int_{\R^3}|f_{s_2}f_{s_2'}|^2)^{1/2}.$$
		Fix $s_j,s_j'$ with distance $d\in[\delta/K,\delta]$ and midpoint $(c_1,c_2)$ between their centers.
		Call $\tilde{s}_j$, $\tilde{s}_j'$ the images of $s_j$, $s_{j}'$ under the map $(\xi,\eta)\mapsto(\frac{\xi-c_1}{d},\frac{\eta-c_2}{d})$. They are $1$-separated rectangles with dimensions $\sim (\delta^2/d,\delta/(Kd))$ lying inside a strip
		whose central line points in the direction  $(\xi_2^*-\xi_1^*,\eta_1^*-\eta_2^*)$. This strip has width $\delta^2/d$, and the corresponding strip on $\ZH$ lies within the $O(\delta^2/d)$-neighborhood of a parabola with curvature $\sim 1$.

		The affine transformation
		$$(\xi,\eta, \gamma)\mapsto (\frac{\xi-c_1}{d},\frac{\eta-c_2}{d},\frac{\gamma-c_1\eta-c_2\xi+c_1c_2}{d^2})$$
		maps $\ZH$ to itself, and $N_{1/R}(\ZH)$ to $N_{1/(Rd^2)}(\ZH)$. Call $g_{\tilde{s}_j}$, $g_{\tilde{s}_{j'}}$  the rescaled versions of $f_{s_j}$, $f_{s_j'}$	according to this map. Their Fourier support lies inside 1-separated subsets of the $O(\delta^2/d)$-neighborhood of a parabola with curvature $\sim 1$. This is because $1/(Rd^2)\lesssim \delta^2/d$, a consequence of our hypothesis $R\gtrsim K\delta^{-3}$. 
		
		We first use bilinear restriction (Proposition \ref{labell}) and split $\tilde{s}_j$, $\tilde{s}_j'$ into $\delta^2/d$-squares $\tilde{t}_j,\tilde{t}_j'$ to get
		$$\int_{\R^3}|g_{\tilde{s}_j}g_{\tilde{s}_j'}|^2\lesssim \sum_{\tilde{t}_j\subset\tilde{s}_j}\sum_{\tilde{t}_j'\subset\tilde{s}_j'}\int_{\R^3}|g_{\tilde{t}_j}g_{\tilde{t}_j'}|^2.$$ 
		For each such pair $(\tilde{t}_j,\tilde{t}_j')$,  we then apply the induction hypothesis. Transversality holds essentially because the absolute value of the slope of the line joining the centers of $\tilde{t}_j,\tilde{t}_j'$	equals the absolute value of the slope of $\ell(\tau_1,\tau_2)$.
        Thus,
		$$\int_{\R^3}|g_{\tilde{s}_j}g_{\tilde{s}_j'}|^2\lesssim $$$$C(\delta^2/d, Rd^2)\Big(\sum_{\tilde{\theta}_j\in \cP_{(Rd^2)^{-1/2}}(\tilde{s}_j)}\big\|g_{\tilde{\theta}_j}\big\|_{L^4(\R^3)}^2\Big)\Big(\sum_{\tilde{\theta}_j'\in \cP_{(Rd^2)^{-1/2}}(\tilde{s}_j')}\big\|g_{\tilde{\theta}_j'}\big\|_{L^4(\R^3)}^2\Big).$$
		Using monotonicity (recall that $d\ge \delta/K$) and \eqref{hguugyrtugyuy5yty}, we may write $$C(\delta^2/d, Rd^2)\le \max_{R'\le R\delta^2}C(\delta K,R')\lesssim K^{O(1)}\max_{R'\le R\delta^2}C(\delta ,R').$$		 Rescaling back (and using that $s_j,s_j'\subset \omega_j$) it follows that
		$$\int_{\R^3}|f_{s_j}f_{s_j'}|^2\lesssim$$$$
		K^{O(1)}\sup_{R'\le R\delta^2}C(\delta , R')\Big(\sum_{{\theta}_j\in \cP_{R^{-1/2}}({\omega}_j)}\big\|f_{{\theta}_j}\big\|_{L^4(\R^3)}^2\Big)^2.$$
		Thus $$\max_{s_1\not\sim s_{1}'\subset \omega_1}\max_{s_2\not\sim s_{2}'\subset \omega_2}\int_{\R^3} |f_{s_1}f_{s_1'}f_{s_2}f_{s_2'}|\lesssim $$
		$$K^{O(1)}\sup_{R'\le R\delta^2}C(\delta , R')\prod_{j=1}^2\Big(\sum_{{\theta}_j\in \cP_{R^{-1/2}}({\omega}_j)}\big\|f_{{\theta}_j}\big\|_{L^4(\R^3)}^2\Big).$$ Summing over $\omega_1, \omega_2$ leads to the second upper bound in \eqref{bfjrhfurhgrugy}. \qedhere
		
	\end{proof}
    \medskip 
	\begin{proof}[Proof of Theorem\ref{dec-thm}]
		Fix $\e>0$. Let $K=R^{\e^{2}}$. We first invoke \eqref{hguugyrtugyuy5yty}
		$$C(1,R)\lesssim R^{O(\e)}C(R^{-\e},R).$$
		We iterate \eqref{bfjrhfurhgrugy} starting with the value $\delta= R^{-\e}$. Each iteration doubles the number of terms. New terms either substantially decrease (and never increase) the value of $\delta$, or substantially decrease (and never increase) the value of $R$.
		
		Each term is iterated until it becomes of the form $C(\Delta, r)$, with $\Delta\le r^{-1/4}$. Proposition \ref{jfiriurgutgur8g} guarantees that each such term contributes $\lesssim 1$.
		
		A term $C(\Delta,r)$ needs further iteration as long as $\Delta>r^{-1/4}$. But since each $\Delta$ satisfies $\Delta\le K^{-1}$ (recall that the initial value satisfies $\Delta=R^{-\e}\le K^{-1}$, and $\Delta$ never gets larger), $\Delta>r^{-1/4}$ implies $r>K\Delta^{-3}$. This is precisely the requirement in Proposition \ref{diowue8tu954iy967u989}, which guarantees that \eqref{bfjrhfurhgrugy} is applicable to $C(\Delta,r)$.  
		
		It remains to understand the number of steps needed for such an iteration to reach a halt and  the accumulation of multiplicative constants. We describe the two extreme scenarios and leave the details for the general case to the reader. First, if \eqref{bfjrhfurhgrugy} only contained the first term, it would need to be iterated $n$ times until $R^{-\e}/K^n\sim R^{-1/4}$. This shows $n\sim \e^{-2}$, and the final multiplicative constant is $(C_2)^{n}\lesssim_\e 1$. If instead \eqref{bfjrhfurhgrugy} only contained the second term, it would need to be iterated $n$ times until $R^{-\e}=(R^{1-2n\e})^{-1/4}$. In this case $n\sim \e^{-1}$, and the corresponding loss is $K^{O(\e^{-1})}=R^{O(\e)}$. In either case, the multi-iteration produces $O_\e(1)$ many terms, and  we are led to the bound
		\[C(1,R)\lesssim_\e R^{O(\e)}. \qedhere\]
		
	\end{proof}
    \medskip

    \subsection{Application to linear decoupling}
    \label{app-linear-dec-sec}
    
    We now reprove the following recent result of Guth, Maldague, and Oh \cite{GMO}.
    They observed that the $\ell^2$ decoupling for $\ZH$ is salvaged if partitions are replaced with appropriate $\log R$-overlapping covers. 
    
    \begin{theorem}
    \label{GMOt}
    Let $\mathcal{R}_R$ be the collection of all dyadic rectangles $\omega$ in $[-1,1]^2$, with sidelength $R^{-1}\le 2^n\le 2$ and area $R^{-1}$. Then for each $f$ Fourier supported on $N_{1/R}(\ZH_{[-1,1]^2})$ we have
    $$\|f\|_{L^4(\R^3)}\lesssim_\e R^\e(\sum_{\omega\in\mathcal{R}_R}\|f_\omega\|_{L^4(\R^3)}^2)^{1/2}.$$
    \end{theorem}
	\begin{proof}
    Let us call $D(R)$ the best constant in the previous inequality. We need to prove $D(R)\lesssim_\e R^{\e}$.
    
    We start with a broad-narrow argument.
    Fix $K=2^m\ge 1$ for some $m$ to be chosen later, and let $\cC_K$ be the partition of $[-1,1]^2$ into $K^{-1}$-squares $\tau$. 
    For each $x\in\R^3$, let $\tau_1$ be the square maximizing $|f_{\tau_1}(x)|$. Triangle's inequality implies that
    $$|f_{\tau_1}(x)|\ge K^{-2}|f(x)|.$$
    Let $\mathcal{S}_{\tau_1}$ consist of those $\tau\in\mathcal{C}_K$ such that both the distance between the $\xi$-coordinates and the $\eta$-coordinates of the centers $c(\tau),c(\tau_1)$  are at least $2/K$. 
    Let us call  $1/K$-transverse any such pair $(\tau,\tau_1)$.
    We let
    $$\mathcal{S}_{big}=\{\tau\in\mathcal{S}_{\tau_1}:\;|f_\tau(x)|\ge \frac12K^{-2}|f(x)|\}.$$
    There are three possibilities.
    \medskip

    {\bf Case 1.} If $|f_{\tau_1}(x)|\ge \frac{1}{100}|f(x)|$, then we write
    \begin{equation}
    \label{kjhhrfhrghurhguhrgiuh1}
    |f(x)|\lesssim \max_{\tau}|f_\tau(x)|\le (\sum_{\tau}|f_{\tau}|^4)^{1/4}.
    \end{equation}
    \medskip

    {\bf Case 2.} If $\mathcal{S}_{big}$ is nonempty, we find that that\begin{equation}
    \label{kjhhrfhrghurhguhrgiuh2}
|f(x)|\lesssim K^4\max_{(\tau_1,\tau_2):\;1/K-\text{transverse}}(f_{\tau_1}(x)f_{\tau_2}(x))^{1/2}.
    \end{equation}
    \medskip

    {\bf Case 3.} Assume $\mathcal{S}_{big}$ is empty and $|f_{\tau_1}(x)|\le \frac1{100}|f(x)|$. Since 
    $$\sum_{\tau\in\mathcal{S}_{\tau_1}}|f_\tau(x)|<\frac12|f(x)|,$$
    it follows that 
    \begin{equation}
    \label{hfhurhfurhfuregurgiuy}
    |\sum_{\tau\in \mathcal{S}\setminus\mathcal{S}_{\tau_1}}f_\tau(x)|\ge \frac12|f(x)|.
    \end{equation}
    Note that $\mathcal{S}\setminus\mathcal{S}_{\tau_1}$ is the union of three (vertical) $(1/K,2)$-rectangles $\omega$ and three (horizontal) $(2,1/K)$-rectangles $\omega$. If we exclude the nine neighbors of $\tau_1$ (itself included), the six rectangles are pairwise disjoint. Since the nine neighbors contribute at most $\frac{9}{100}|f(x)|$, triangle's inequality and \eqref{hfhurhfurhfuregurgiuy} shows that one of the six $\omega$ satisfies
    $|f(x)|\le 100|f_\omega(x)|.$ We summarize our findings as follows
    \begin{equation}
    \label{kjhhrfhrghurhguhrgiuh3}
    |f(x)|\lesssim \max_{\omega:\;(1/K,2)-\text{rectangle or }\atop{(2,1/K)-\text{rectangle}}}|f_\omega(x)|\le (\sum_{\omega}|f_\omega|^4)^{1/4}.
    \end{equation}
    We mention that all implicit constants in the inequalities from the three cases are independent of $K$. Let us call $C$ the maximum of these constants. 
    \smallskip

    If \eqref{kjhhrfhrghurhguhrgiuh1} holds for each $x$, rescaling by $(2K,2K,4K^2)$ leads to the inequality
    $$D(R)\le CD(R/4K^2).$$
	If \eqref{kjhhrfhrghurhguhrgiuh2} holds for each $x$, then Theorem \ref{dec-thm} implies that
    $$D(R)\lesssim_\e K^{O(1)}R^{\e}.$$
    If \eqref{kjhhrfhrghurhguhrgiuh3} holds for each $x$, we rescale each $\ZH_{\omega}$ with horizontal $\omega$ by $(1,2K,2K)$ and each $\ZH_{\omega}$ for vertical $\omega$ by $(2K,1,2K)$. Note that these non-isotropic dilations leave $\ZH$ invariant. In this case we get
    $$D(R)\le CD(R/2K).$$
   It is precisely this case that explains the need for the collection $\mathcal{P}_R$ of all rectangles in the definition of $D(R)$.

Overall, we have the inequality
$$D(R)\le C(D(R/2K)+D(R/4K^2))+C_\e K^{C}R^\e.$$
We may now pick $K=(100C)^{1/\e}$. Iterating the above inequality proves
$D(R)\lesssim_\e R^\e$. \qedhere

\end{proof}

    \bigskip 
	
\section{Bilinear refined decoupling}

    Throughout this section, we fix $\e$ and let $K_1=R^{\e^{6}}$, $K_2=R^{\e^{4}}$, $K_3=R^{\e^{2}}$. Note that $1\ll K_1\ll K_2\ll K_3\ll R^\e$.
	$K_1$ will be used to enforce the broad-narrow dichotomy, $K_2$ will be the eccentricity of the rectangles, $K_3$ we be a threshold factor that enforces the stopping time.
	
	We start by recalling a few tools from the previous section, adapted to the new context. Definition \ref{transversality} introduced transverse squares that are separated by $\sim 1$. We will now encounter pairs of squares that are separated by $\ll 1$. 
    \begin{definition}
    We will refer to a pair of squares in $[-1,1]^2$ as being in {\bf general position} if the line joining their centers has slope of absolute value  $\sim 1$. A thin rectangle is in general position if its long central line satisfies the same property.  
    \end{definition}
    Throughout this section, $Q$ will refer to an arbitrary $\sqrt{R}$-ball in $\R^3$. We will consider various pairs of  disjoint squares with diameter $\sim r$ and separation $d$.
    Unless stated otherwise, it will be implicitly assumed that $d\ge r$. This separation condition records the fact that the squares are not neighbors, and is preserved for pairs consisting of their descendants, as  $d$ does not decrease.
	\begin{lemma}
		\label{main-toolff}
        Assume $\widehat{f}$ is supported on 
        $N_{1/\sqrt{R}}(\ZH)$.
		Consider a pair of $r$-squares $(\alpha_1,\alpha_2)$ in general position, with centers at distance $d$ satisfying $dR^{1/2}r \gtrsim K_2$ (in addition to the implicit assumption that $d\ge r$). We have 		$$\int_Q|f_{\alpha_1}f_{\alpha_2}|^2\sim \int_Q\sum_{\omega_1\subset \alpha_1}|f_{\omega_1}|^2\sum_{\omega_2\subset \alpha_2}|f_{\omega_2}|^2,$$
		where $\omega_j$ are $(r/K_2,r)$-rectangles in general position.
	\end{lemma}
    \begin{proof} The diameter of $Q$ is consistent with the thickness $1/\sqrt{R}$ of the frequency support, so $Q$ can be discarded using the Uncertainty Principle.  We now use Lemma \ref{main-tool} as in Step 4 of the proof of Proposition \ref{diowue8tu954iy967u989}. The affine transformation
		$$L(\xi,\eta, \gamma)= (\frac{\xi-c_1}{d},\frac{\eta-c_2}{d},\frac{\gamma-c_1\eta-c_2\xi+c_1c_2}{d^2})$$
		maps $N_{1/\sqrt{R}}(\ZH)$ to $N_{1/(\sqrt{R}d^2)}(\ZH)$. Under this map, the frequency support of $f_{\alpha_i}$ is mapped to $N_{R^{-1/2}d^{-2}}(\ZH_{\tau_i})$, for some $\tau_1,\tau_2\subset [-2,2]^2$. This lies inside the $O(R^{-1/2}d^{-2})$-neighborhood of a plane $\pi_i$. Our hypothesis implies that $R^{-1/2}d^{-2}\lesssim \Delta:=\frac{r}{K_2d}$. We apply Lemma \ref{main-tool} with this $\Delta$, an then rescale back using $L^{-1}$. \qedhere
    
\end{proof}
	We note that this result proves an equivalence (double inequality) between the  uncoupled term on the left, and the decoupled term on the right. Thus, this reverse square function estimate is reversible; terms on the right hand side  may be conveniently recoupled. See \eqref{NB}. This feature is not crucial to the argument, but it leads to simplifications that have esthetic value. 
	
	The next result is the broad-narrow decomposition for each term in Lemma \ref{main-toolff}.
	\begin{lemma}
		Consider the family of $K_1$ many  $(r/K_2,r/K_1)$-rectangles $s_i$ partitioning  $\omega_i$. Then 
		\begin{align*}
			\int_Q|f_{\omega_1}|^2|f_{\omega_2}|^2\lesssim &\int_Q\sum_{s_1\subset \omega_1}|f_{s_1}|^2\sum_{s_2\subset \omega_2}|f_{s_2}|^2\\&+K_1^{O(1)}\max_{s_1\not\sim s_1'\subset \omega_1}\max_{s_2\not\sim s_2'\subset \omega_2}\int_Q|f_{s_1}f_{s_1'}f_{s_2}f_{s_2'}|.
		\end{align*}
	\end{lemma}
	We note that 
	$$\max_{s_1\not\sim s_1'\subset \omega_1}\max_{s_2\not\sim s_2'\subset \omega_2}\int_Q|f_{s_1}f_{s_1'}f_{s_2}f_{s_2'}|\le \int_Q|g_{\omega_1}g_{\omega_2}|^2,$$
	where
    \begin{equation}
    \label{function-g}
        g_{\omega_i}=(\sum_{s_i\not\sim s_i'\subset \omega_i}|f_{s_i}f_{s_i'}|)^{1/2}.
    \end{equation}
The least favorable scenario for pairs $s_i\not\sim s_i'$ is when $s_i,s_i'$ are almost adjacent (their centers are separated by only $2r/K_1$). To simplify notation (when it comes to rescaling), we will assume that the summation in the definition of $g_{\omega_i}$ is restricted to such pairs. 
    
	We next combine  these two lemmas with recoupling. We mention that recoupling is only used to keep the argument more elegant. It is not an essential tool. 
	\begin{lemma}
	Consider a pair $(\alpha_1,\alpha_2)$ of $r$-squares  in general position, with centers at distance $d$ satisfying $dR^{1/2}r\gtrsim  K_2$. Then	$$\int_{Q}|f_{\alpha_1}f_{\alpha_2}|^2\lesssim$$	
		\begin{equation}
			\label{NB}
			\max\left\{\int_{Q}\sum_{\beta_1\subset \alpha_1}|f_{\beta_1}|^2\sum_{\beta_2\subset \alpha_2}|f_{\beta_2}|^2 ,\;K_1^{O(1)}\sum_{\omega_1\subset \alpha_1}\sum_{\omega_2\subset \alpha_2}\int_{Q}|g_{\omega_1}g_{\omega_2}|^2\right\}.\end{equation}
		The sum in the first term is over  $r/K_1$-squares $\beta_i$ partitioning $\alpha_i$.
	\end{lemma}
	\begin{proof}
	Use recoupling to reassemble rectangles $s_i$ into squares $\beta_i$. \qedhere
	
	\end{proof}
	
	\begin{definition}Consider a pair of $r$-squares $(\alpha_1,\alpha_2)$ in general position, with centers at distance $d$ satisfying 
	\begin{equation}
	\label{firuegrtug8tu8uh8u8u}
	dR^{1/2}r \gtrsim K_2.
	\end{equation}
		We call the pair $(\alpha_1,\alpha_2)$ narrow/broad relative to $Q$ if the first/second term in \eqref{NB} dominates.
	\end{definition}
	The phrase ``relative to $Q$" will be omitted, when $Q$ is clear from the context. We emphasize that we require \eqref{firuegrtug8tu8uh8u8u} to hold in order for a pair to be labeled either narrow or broad.

	\begin{lemma}
		\label{procC}	
		Assume $s_i,s_i'$ are almost adjacent rectangles inside some $(r/K_2,r)$-rectangle $\omega_i$ in general position. Assume $r^2\sqrt{R}\gtrsim K_1K_2$. Then
		$$\int_{Q}|f_{s_i}f_{s_i'}|^2\lesssim \int_{Q}\sum_{t_i\subset s_i}|f_{t_i}|^2\sum_{t_i'\subset s_i'}|f_{t_i'}|^2,$$
		where $t_i \;(t_i')$ are $r/K_2$-squares  partitioning $s_i\; (s_i')$.		
	\end{lemma}
	\begin{proof}
	Use rescaling (via a map called  $T$) by the factor $K_1/r$, centered at the midpoint of $(s_i,s_i')$. Then $\ZH_{T(\omega_i)}$ lies within the $O(K_1/K_2)$-neighborhood of a parabola (with curvature $\sim 1$). Also, $T(s_i)$, $T(s_i')$ are 1-separated. The ball $Q$ is mapped to a set that is efficiently covered by $\sqrt{R}r^2/K_1^2$- balls. We note that $\sqrt{R}r^2/K_1^2\gtrsim K_2/K_1$ and we apply bilinear restriction (Proposition \ref{labell}) on balls of radius $K_2/K_1$. Then we rescale back.
    \qedhere
	  
	\end{proof}
	\begin{proposition}[Rescaling $C(R)$]
		\label{jijgutjhguthughru}
		Let $r\sqrt{R}\ge K_1$.
		Assume $\omega$ is in general position, with dimensions $(r/K_2,r)$.	
		Assume $s,s'$ are almost adjacent $(r/K_2,r/K_1)$-rectangles inside $\omega$. 
		Suppose $X$ is a collection of $\sqrt{R}$-squares each of which intersects at most $M$ many $R$-tubes from $\ZT(s)$ and at most $M'$ many $R$-tubes from $\ZT(s')$. Then	
		$$\int_X|f_{s}f_{s'}|^2\lesssim  C(Rr^2/K_1^2) (MM')^{1/2}\Big(\sum_{T\in\ZT(s)}\big\|f_{T}\big\|_4^4\Big)^{1/2}\Big(\sum_{T\in\ZT(s')}\big\|f_{T}\big\|_4^4\Big)^{1/2}.$$	
	\end{proposition}
	\begin{proof}
		Place $s,s'$ inside $2r/K_1$-separated $r/K_1$-squares. Rescale them by a factor $K_1/r$. The resulting squares are transverse.  Cover $X$ by $(\sqrt{R},\sqrt{R}, \sqrt{R}K_1/r)$- tubes. These tubes become $r\sqrt{R}/K_1$-squares under parabolic rescaling. 
	\end{proof}
    \smallskip

	We next present the key technical tool that replaces a layer of terms $g_\omega$ with another layer of smaller scale. 
	\begin{proposition}
	\label{pmainnnnnnnnnn}	
		Let $r\gtrsim K_3^{-1}$. Fix an $R^{1/2}$-ball $Q$.
		Assume $\{\omega\}$ is a collection of pairwise disjoint $(r/K_2,r)$-rectangles in general position. Then  one of the following is true:
		
		(1) there is  $R^{-1/2}K_1\lesssim r'\le r/K_2$
		and a collection of pairwise disjoint $(r'/K_2,r')$-rectangles  $\omega'\subset \bigcup\omega$ in general position such that 
		\begin{equation}
		\label{irojiojgiortjgiortjghio}
		\sum_{\omega}\|g_\omega\|_{L^4(Q)}^2\lesssim_\e (\frac{r}{r'})^{100\frac{\log{K_1}}{\log K_2}} \sum_{\omega'}\|g_{\omega'}\|_{L^4(Q)}^2,\end{equation}
		
		(2) we have
		\begin{equation}
			\label{dhdgfgjhuh3}\sum_{\omega}\|g_\omega\|_{L^4(Q)}^2\lesssim_\e R^\e \sum_{\theta\subset \bigcup\omega}\|f_{\theta}\|_{L^4(Q)}^2,
		\end{equation}
		where the last sum is over a partition of $\bigcup\omega$ into $R^{-1/2}$-squares $\theta$.
	\end{proposition}
	\begin{proof}
		For each $\omega$, fix $r/K_1$-separated $(r/K_2,r/K_1)$-rectangles $s_1(\omega),s_2(\omega)\subset\omega$ such that
		$$\|g_\omega\|_{L^4(Q)}^2\le K_1^{O(1)}(\int_Q|f_{s_1(\omega)}f_{s_2(\omega)}|^2)^{1/2}.$$
		Note that 
		\begin{equation}
		\label{keychoice}
		r\gtrsim K_3^{-1}\implies r^2\sqrt{R}\gtrsim K_1K_2^2.
	\end{equation}
		We may thus apply Lemma \ref{procC}	to each $\omega$ to get
		\begin{equation}
			\label{dhdgfgjhuh1}
			\sum_{\omega}\|g_\omega\|_{L^4(Q)}^2\lesssim K_1^{O(1)}\sum_{\omega}\; (\sum_{\alpha_1\subset s_1(\omega)}\sum_{\alpha_2\subset s_2(\omega)}\int_{Q}|f_{\alpha_1}f_{\alpha_2}|^2)^{1/2},
		\end{equation}
		where the pairs $(\alpha_1,\alpha_2)$ consist of $r/K_1$-separated $r/K_2$-squares inside $\omega$, in general position.
		We note that due to \eqref{keychoice}, each such pair satisfies \eqref{firuegrtug8tu8uh8u8u}, so it is either narrow or broad. At the expense of a multiplicative  factor of 2,  we may assume all pairs $(\alpha_1,\alpha_2)$ are of the same type. Let us explain why. We have
		$$\int_{Q}|f_{\alpha_1}f_{\alpha_2}|^2\le \max\{N(\alpha_1,\alpha_2),B(\alpha_1,\alpha_2)\},$$
		where $N, B$ denote the two terms in \eqref{NB}. We use the abstract inequality ($S(\omega)$ is any collection of pairs)
		$$\sum_{\omega}\;\big(\sum_{(\alpha_1,\alpha_2)\in S(\omega)}\max\{N(\alpha_1,\alpha_2),B(\alpha_1,\alpha_2)\}\big)^{1/2}\le $$
		\begin{equation}
		\label{factor2}
		2\max\left \{\sum_{\omega}\;\big(\sum_{(\alpha_1,\alpha_2)\in S(\omega)}N(\alpha_1,\alpha_2)\big)^{1/2},\;\sum_{\omega}\; \big(\sum_{(\alpha_1,\alpha_2)\in S(\omega)}B(\alpha_1,\alpha_2)\big)^{1/2}\right\}.
		\end{equation}
	\medskip
    
	    \noindent {\bf Case (a).} Let us assume all $(\alpha_1,\alpha_2)$ are narrow. We process each pair  and find 
		$$\int_{Q}|f_{\alpha_1}f_{\alpha_2}|^2\le C \int_{Q}\sum_{\beta_1\subset \alpha_1}|f_{\beta_1}|^2\sum_{\beta_2\subset \alpha_2}|f_{\beta_2}|^2,
		$$
		where $\beta_i$ are  $r/(K_1K_2)$-squares partitioning $\alpha_i$.
		
		By the same principle mentioned above, the pairs $(\beta_1,\beta_2)$ can also be assumed to all (this means all pairs corresponding to all $(\alpha_1,\alpha_2)$ and all $\omega$) be either narrow or broad. Let us see what happens if the streak of narrow terms continues for $m$ steps. 
        Since we are in Case (a), we know $m\ge 1$. 
        We run this streak for as long as possible. 
        At the end of it, we are left with the inequality
		\begin{equation}
			\label{dhdgfgjhuh2}\int_{Q}|f_{\alpha_1}f_{\alpha_2}|^2\le C^{m} \int_{Q}\sum_{\gamma_1\subset \alpha_1}|f_{\gamma_1}|^2\sum_{\gamma_2\subset \alpha_2}|f_{\gamma_2}|^2,
		\end{equation}
		where $\gamma_i$ are $r/K_2(K_1)^m$-squares partitioning $\alpha_i$.	The value $m$ is the same for all $(\alpha_1,\alpha_2)$ corresponding to all $\omega$. Moreover, one of two things must happen.
	\medskip
		
	    {\bf Case (a1).} We have essentially reached the bottom scale $R^{-1/2}$. More precisely, the scale $r_1=r/K_2(K_1)^m$ of the terminal squares $\gamma_i$ satisfies $R^{-1/2}\lesssim r_1\lesssim K_1K_2K_3R^{-1/2}$. The choice of the cutoff $K_1K_2K_3R^{-1/2}$ is informed by the necessity of \eqref{ejfirueueg8uu09609y09568y09} being true while $r_1\gtrsim K_1K_2K_3R^{-1/2}$.

			Then  \eqref{dhdgfgjhuh3} follows by combining \eqref{dhdgfgjhuh1}, \eqref{dhdgfgjhuh2}, Minkowski's inequality and the triangle inequality 
		$|f_{\gamma_i}|\le \sum_{\theta\subset \gamma_i}|f_\theta|$.
		The triangle inequality produces the loss 
		\begin{equation}
		\label{cjrecfuugug965uy09u}
		(K_1K_2K_3)^{O(1)}\lesssim_\e R^\e.
		\end{equation}
		Since $K_1^m\lesssim R^{1/2}$, the loss $C^m$ in \eqref{dhdgfgjhuh2} is $O(R^{\frac{100}{\log K_1}})=O_\e(1)$. Also, we lose one factor 2 in \eqref{factor2} for each of the $m$ steps, but this is again acceptable.
	\medskip
		
		{\bf Case (a2).} The other possibility is that the final scale $r_1=r/K_2(K_1)^m$ satisfies $r_1\gtrsim K_1K_2K_3 R^{-1/2}$. Let us explain the reason why the streak must end at such an early stage. Throughout this streak, the distance between new pairs of squares does not decrease. Thus, the distance $d_1$ between terminal pairs $(\gamma_1,\gamma_2)$ satisfies $d_1\ge r/K_1$.
		Using these and the fact that $r\gtrsim K_3^{-1}$ implies that
		\begin{equation}
		\label{ejfirueueg8uu09609y09568y09}
		d_1 R^{1/2}r_1\gtrsim K_2.
		\end{equation}
		Thus, according to \eqref{firuegrtug8tu8uh8u8u},  $(\gamma_1,\gamma_2)$ is  either narrow or broad. But since the narrow streak came to a halt, the pair must be broad. Thus 
		\begin{equation}
        \label{vjgurtiyo0uoj0-ok-}
			\int_{Q}|f_{\gamma_1}f_{\gamma_2}|^2\lesssim K_1^{O(1)}\sum_{\omega_1\subset \gamma_1}\sum_{\omega_2\subset \gamma_2}\int_{Q}|g_{\omega_1}g_{\omega_2}|^2.
		\end{equation}
		Here $\omega_i$ are $(r_1/K_2,r_1)$-rectangles in general position. Their initial orientation is decided not just by individual $\gamma_i$, but by the pair $(\gamma_1,\gamma_2)$. 
        However, we note the following.
        Since $\gamma_i\subset \alpha_i\subset \omega$ and $\dist(\alpha_1,\alpha_2)\ge r/K_1$, the directions of the line $\ell(\gamma_1,\gamma_2)$ joining their centers differs by  $\le r(K_2)^{-1}/r(K_1)^{-1}=K_1/K_2$ from the direction of the central line $\ell_\omega$ of $\omega$.
        Since the eccentricity of $\omega_i$ is $1/K_2$, we may arrange that the orientation of $\omega_i$ is universal. 
        More precisely, we may take each $\omega_i$ to point in the direction perpendicular to $\ell_\omega$. This will come at the expense of the factor $\frac{K_1/K_2}{1/K_2}=K_1$, which fits well into the acceptable loss for a broad step. 
		
		When combining \eqref{dhdgfgjhuh1},  \eqref{dhdgfgjhuh2} and \eqref{vjgurtiyo0uoj0-ok-} we get
		\begin{align*}
			\|g_\omega\|_{L^{4}(Q)}^2 &\lesssim C^mK_1^{O(1)}(\sum_{\omega_1\subset s_1(\omega)}\sum_{\omega_2\subset s_2(\omega)}\int_{Q}|g_{\omega_1}g_{\omega_2}|^2)^{1/2}&\\&=C^mK_1^{O(1)}(\int_{Q}\sum_{\omega_1\subset s_1(\omega)}|g_{\omega_1}|^2\sum_{\omega_2\subset s_2(\omega)}|g_{\omega_2}|^2)^{1/2}\\&\lesssim_\e K_1^{O(1)}(\int_{Q}(\sum_{\omega'\subset \omega}|g_{\omega'}|^2)^2)^{1/2}.
		\end{align*}
		Here $\omega'$ is simply the generic notation for either $\omega_1$ or $\omega_2$.
        In the last step, we dispense with bilinearity between $s_1(\omega)$ and $s_2(\omega)$, as each term $g_{\omega'}$ encodes new transversality. 
        Finally, Minkowski's inequality and summation over $\omega$ lead to 
		$$\sum_{\omega}\|g_\omega\|_{L^4(Q)}^2\lesssim_\e K_1^{100}\sum_{\omega}\sum_{\omega'\subset \omega}\|g_{\omega'}\|_{L^4(Q)}^2.$$
		Note that the rectangles $\omega'$ are pairwise disjoint. This is because all $\omega$ are pairwise disjoint,  all $\gamma_i\subset s_i(\omega)$ are pairwise disjoint, and all $\omega_i\subset \gamma_i$ are pairwise disjoint.
        \smallskip
		
		We are at the end of Case (a2). We may take $r'=r_1$ and we are done, as $r'\le r/(K_1K_2)\le r/K_2$, and thus $K_1\le (r/r')^{\frac{\log K_1}{\log K_2}}$.
		\bigskip

	    \noindent	{\bf Case (b).} Assume all pairs $(\alpha_1,\alpha_2)$ are broad. Then, simply by definition, we get
		$$\int_{Q}|f_{\alpha_1}f_{\alpha_2}|^2\lesssim K_1^{O(1)}\sum_{\omega_1\subset\alpha_1}\sum_{\omega_2\subset \alpha_2}\int_{Q}|g_{\omega_1}g_{\omega_2}|^2.$$
		Here $\omega_i$ are $(r'/K_2,r')$-rectangles, where $r'=r/K_2$.
		As explained in the previous case, the orientation of $\omega_i$ is perpendicular to $\ell_\omega$ (the orientation of the parent rectangle). When combined with \eqref{dhdgfgjhuh1} this leads to
		\begin{align*}
			\|g_\omega\|_{L^4(Q)}^2&\lesssim K_1^{O(1)}(\sum_{\alpha_1\subset s_1(\omega)}\sum_{\alpha_2\subset s_2(\omega)}\sum_{\omega_1\subset\alpha_1}\sum_{\omega_2\subset \alpha_2}\int_{Q}|g_{\omega_1}g_{\omega_2}|^2)^{1/2}\\&\lesssim  K_1^{O(1)}(\int_{Q}\sum_{\omega_1\subset s_1(\omega)}|g_{\omega_1}|^2\sum_{\omega_2\subset s_2(\omega)}|g_{\omega_2}|^2)^{1/2}\\&\le K_1^{O(1)}(\int_{Q}(\sum_{\omega'\subset \omega}|g_{\omega'}|^2)^2)^{1/2}\\&\le (r/r')^{100\log K_1/\log K_2} \sum_{\omega'\subset \omega}\|g_{\omega'}\|_{L^4(Q)}^2. 
		\end{align*}
		Here $\omega'$ are $(r'/K_2,r')$-rectangles. We are done in this case, too, by summing this inequality over all $\omega$.  \qedhere

	\end{proof}

	\begin{remark}[The value and the role of $K_3$]
    \label{remark-K-3}
    \rm 
    
	Let us briefly recap the previous argument, in order to explain our choice of the stopping time $K_3$. The first time we used $K_3$ was in \eqref{keychoice}. This inequality by itself would allow $K_3$ to be as large as $\approx R^{1/4}$. However, \eqref{cjrecfuugug965uy09u}	forces $K_3\approx 1$. We recall that the cutoff $K_1K_2K_3$ appearing in \eqref{cjrecfuugug965uy09u} is enforced by \eqref{ejfirueueg8uu09609y09568y09}.
	
	The final induction on scales argument will show that $K_3$ has to be slightly larger than $K_2$. See \eqref{ejiorjfiougioturgiou}.
	\end{remark}
    
	\begin{remark}[$O_\e(1)$ many choices for partitions]
	\label{poi5tu549u5609ug09uy09}
    \rm

	An inspection of the argument shows  that if \eqref{irojiojgiortjgiortjghio}	happens, then $\{\omega'\}$ form a partition of $\cup_{\omega}(s_1(\omega)\cup s_2(\omega))$. This partition may depend on $Q$. However, it is not difficult to see that there are only $O_\e(1)$ such partitions that may arise for various $Q$. Indeed, the partition is entirely determined by the scale $r'$ of the $\omega'$, which takes the form $r/K_2(K_1)^m$, for some $m\ge 0$ ($m=0$ in case (b) of the proof). Since we also have $r'\gtrsim R^{-1/2}$, it follows that $m=O_\e(1)$.
	\end{remark}
    
    \smallskip 
	The following result holds by simply iterating the previous proposition.
	\begin{proposition}
	\label{12345}	
		Let $r\le 1$. Fix an $R^{1/2}$-ball $Q$.
        Assume $\{\omega\}$ is a collection of  $(r/K_2,r)$-rectangles in general position. 
        Then there is a scale  $K_1R^{-1/2}\lesssim r'\lesssim K_3^{-1}$ and there is a  collection $\{\omega'\}$ consisting of pairwise disjoint $(r'/K_2,r')$-rectangles  $\omega'\subset \bigcup \omega$ in general position such that 
		$$\sum_{\omega}\|g_\omega\|_{L^4(Q)}^2\lesssim_\e \sum_{\omega'}(\frac{r}{r'})^{100\frac{\log{K_1}}{\log K_2}} \|g_{\omega'}\|_{L^4(Q)}^2+R^\e \sum_{\theta\subset \bigcup \omega}\|f_{\theta}\|_{L^4(Q)}^2,
		$$
		where the last sum is over a partition of $\bigcup \omega$ into $R^{-1/2}$-squares $\theta$.
	\end{proposition}
	\begin{proof}
		If $r\le K_3^{-1}$ we may take $r'=r$ and $\{\omega'\}=\{\omega\}$. Otherwise apply Proposition \ref{pmainnnnnnnnnn} to the collection $\{\omega'\}$. Repeat this process until either the scale $r'$ gets smaller than $K_3^{-1}$ and the first term dominates, or the scale gets down to $R^{-1/2}$ and the second term dominates. \qedhere

	\end{proof}
    \begin{remark}[Tree depth and $O_\e(1)$ many partitions]
    \label{pfuf}
    \rm 
    For a given collection $\{\omega\}$, the collection $\{\omega'\}$ depends on $Q$. However, there are only $O_\e(1)$ possible collections that may arise this way. Indeed, since each application of Proposition \ref{pmainnnnnnnnnn} decreases the scale by a multiplicative factor of at least $K_2$, the resulting tree has $O_\e(1)$ many layers. 
    As observed in Remark \ref{poi5tu549u5609ug09uy09}, each layer determines the next layer up to $O_\e(1)$ many choices. Then of course, $O_\e(1)^{O_\e(1)}=O_\e(1)$. 
	
    \end{remark}		
    
The next result serves the purpose of separating the contributions from the initial pair of transverse squares $\Omega_1,\Omega_2$. This is necessary due to the presence of the geometric average in the intended upper bound \eqref{ref-dec-esti}.	
\begin{proposition}
\label{rejgiotjgiotjio}	
Let $\Omega_1,\Omega_2$ be transverse $1/K_2$-squares. Fix some $R^{1/2}$-square $Q$. Then
one of the following is true:

(1) there is $r\gtrsim K_1 R^{-1/2}$ and a family of pairwise disjoint $(r/K_2,r)$-rectangles $\omega_i\subset \Omega_i$, in general position, such that
\begin{equation}
	\label{cccc1}
\int_{Q}|f_{\Omega_1}f_{\Omega_2}|^2\lesssim K_1^{O(1)}\sum_{\omega_1\subset \Omega_1}\|g_{\omega_1}\|_{L^4(Q)}^2\sum_{\omega_2\subset \Omega_2}\|g_{\omega_2}\|_{L^4(Q)}^2,
\end{equation}

(2) we have 
\begin{equation}
\label{cccc2}
\int_{Q}|f_{\Omega_1}f_{\Omega_2}|^2\lesssim K_2^{O(1)}
\int_Q\sum_{\theta\subset \Omega_1}|f_\theta|^2
\sum_{\theta\subset \Omega_2}|f_{\theta}|^2,
\end{equation}
where $\theta$ are $R^{-1/2}$-squares.
\end{proposition}	
\begin{proof}
The argument is an easier version of the one in 	Proposition \ref{pmainnnnnnnnnn}. Matters related to growth of constants and orientation of rectangles are identical.

If the pair $(\Omega_1,\Omega_2)$ is broad, then we may take $r=1/K_2$ in \eqref{cccc1}. Indeed, first by definition, then by H\"older's inequality followed by Minkowski's inequality we have
\begin{align*}
\int_{Q}|f_{\Omega_1}f_{\Omega_2}|^2&
\lesssim K_1^{O(1)}\int_Q\sum_{\omega_1\subset\Omega_1}|g_{\omega_1}|^2\sum_{\omega_2\subset \Omega_2}|g_{\omega_2}|^2\\&\lesssim K_1^{O(1)}(\int_Q(\sum_{\omega_1\subset\Omega_1}|g_{\omega_1}|^2)^2)^{1/2}(\int_Q(\sum_{\omega_2\subset \Omega_2}|g_{\omega_2}|^2)^2)^{1/2}
\\&\lesssim K_1^{O(1)}\sum_{\omega_1\subset \Omega_1}\|g_{\omega_1}\|_{L^4(Q)}^2\sum_{\omega_2\subset \Omega_2}\|g_{\omega_2}\|_{L^4(Q)}^2.
\end{align*} 	
Here $\omega_i$ are $(1/K_2^2,1/K_2)$-rectangles.

Let us now assume $(\Omega_1,\Omega_2)$ is narrow. In fact, let us assume that the narrow streak persists for $m$ steps ($m\ge 1$).  After these $m$ steps we have the upper bound
$$
\int_{Q}|f_{\Omega_1}f_{\Omega_2}|^2
\lesssim \sum_{\gamma_1\subset \Omega_1}\sum_{\gamma_2\subset \Omega_2}\int_{{Q}}|f_{\gamma_1}f_{\gamma_2}|^2,$$
where $\gamma_i$ are $1/(K_2K_1^m)$-squares. Note that the distance between pairs remains $\sim 1$, so the  hypothesis $dR^{1/2}r \gtrsim K_2$ in \eqref{firuegrtug8tu8uh8u8u}   is satisfied for $r$ all the way down to the smallest scale $r\sim K_2R^{-1/2}$.

 The streak ends because of  one of two reasons. Either the pairs $(\gamma_1,\gamma_2)$ are broad, in which case we end the argument by repeating the computations from the previous case, with $(\Omega_1,\Omega_2)$ replaced with $(\gamma_1,\gamma_2)$. We get \eqref{cccc1} with $r=1/(K_2K_1^m)$. The other possibility is that the scale of $\gamma_i$ is $\sim K_2R^{-1/2}$, in which case we have \eqref{cccc2} (via another application of the triangle inequality). \qedhere

\end{proof}

We next analyze the case when \eqref{cccc1} holds. The next result will then be applied separately to $\Omega=\Omega_1$ and $\Omega=\Omega_2$.

\begin{proposition}
\label{fjgjirtgihiythiuij09u}	
Let $X$ be a collection of $R^{1/2}$-balls $Q$. Let $\Omega\subset [-1,1]^2$ be a square.
Let $f=\sum_{T\in\ZT}f_T$ be a sum of scale $R$ wave packets so that ${\rm supp}(\wh f)\subset N_{R^{-1}}(\ZH_{\Omega})$. 
Suppose there is $M\geq1$ such that each  $Q\subset X$ intersects  at most $ M$ many $R$-tubes from $\ZT$.
Let $r\le 1$. 
Consider a collection of pairwise disjoint $(r/K_2,r)$-rectangles $\omega\subset \Omega$ in general position. 
Then ($g$ depends on $f$, as in \eqref{function-g})	
$$\sum_{Q\subset X}(\sum_{\omega}\|g_{\omega}\|_{L^4(Q)}^2)^{2}\lesssim $$$$\left((\log R)^{O(1)}\sup_{r'\lesssim K_3^{-1}}(\frac{r}{r'})^{200\frac{\log{K_1}}{\log K_2}}C(R(r')^2/K_1^2)+R^\e\right) M\sum_{T\in\ZT}\big\|f_{T}\big\|_4^4.$$

\end{proposition}
\begin{proof}
We apply Proposition \ref{12345} to each $Q$.
We get a scale $K_1R^{-1/2}\lesssim r'\lesssim K_3^{-1}$ and a collection of  pairwise disjoint $(r'/K_2,r')$-rectangles  $\omega'\subset \Omega$ in general position such that 
$$\sum_{\omega}\|g_\omega\|_{L^4({Q})}^2\lesssim_\e \sum_{\omega'}(\frac{r}{r'})^{100\frac{\log{K_1}}{\log K_2}} \|g_{\omega'}\|_{L^4({Q})}^2+R^\e \sum_{\theta\subset \Omega}\|f_{\theta}\|_{L^4(Q)}^2,
$$
where the last sum is over a partition of $\Omega$ into $R^{-1/2}$-squares $\theta$.
\smallskip

We first note the upper bound for the second term
$$\sum_{Q\subset X}(\sum_{\theta\subset \Omega}\|f_{\theta}\|_{L^4({Q})}^2)^2\lesssim M\sum_{\theta\subset \Omega}\|f_{\theta}\|_{L^4({X})}^4\lesssim M\sum_{\theta\subset \Omega}\|f_{\theta}\|_{L^4({\R^3})}^4\sim M\sum_{T\in\ZT}\big\|f_{T}\big\|_4^4.$$
\smallskip

For the first term, we first pigeonhole and assume each $Q$ has the same collection $\{\omega'\}$, cf. Remark \ref{pfuf}. Via another pigeonholing, we  may also assume that, for some fixed $N$,  each $Q$ receives a $(\log R)^{-O(1)}$-fraction of the contribution to the integral from $\sim M/N$ tubes, from each of $\sim N$ many rectangles $\omega'$. For such a pair, we write $Q\sim\omega'$.
Then
\begin{align*}
\sum_{Q}\;(\sum_{\omega'}\|g_{\omega'}\|_{L^4(Q)}^2)^2&\lesssim (\log R)^{O(1)} N\sum_{Q}\sum_{\omega'\sim Q}\|g_{\omega'}\|_{L^4(Q)}^4\\&=(\log R)^{O(1)} N\sum_{\omega'}\|g_{\omega'}\|_{L^4(\cup_{Q\sim\omega'}Q)}^4.
\end{align*}
By Proposition \ref{jijgutjhguthughru} (with geometric averages replaced by sums), $$\|g_{\omega'}\|_{L^4(\cup_{Q\sim\omega'}Q)}^4\lesssim C(R(r')^2/K_1^2) K_1^{O(1)}M/N\sum_{T\in\ZT_{\omega'}}\|f_{T}\big\|_4^4.$$
Finally, we  combine the last two inequalities and sum over $\omega'$, noting that the collections $\ZT_{\omega'}$ are pairwise disjoint, since the rectangles $\omega'$ are pairwise disjoint. \qedhere

\end{proof}

\smallskip

We combine the previous two propositions to prove the following theorem.
\begin{theorem}
\label{jregiurtiguiuhyuh8}	
We have$$C(R)\lesssim_\e$$
\begin{equation}
\label{fiojugurtuhgiyt9hi0967i}	
 (K_1K_2)^{O(1)}\left((\log R)^{O(1)}\sup_{r'\lesssim K_3^{-1}}(\frac{1}{r'})^{200\frac{\log{K_1}}{\log K_2}}C(R(r')^2/K_1^2)+R^\e\right).
\end{equation}
\end{theorem}
\begin{proof}
Let $\tau_1,\tau_2$ be arbitrary transverse squares, and we let $f_1,f_2,M_1,M_2,X$ be as in the definition of $C(R)$. We partition $\tau_i$ into $1/K_2$-squares $\Omega_i$, and use the triangle inequality to write
\begin{equation}
\label{f54kgjrihyihiyhiytuh}
\int_{X}|f_1f_2|^2\le K_2^{O(1)}\sum_{\Omega_1,\Omega_2}\int_X|f_{\Omega_1}f_{\Omega_2}|^2.\end{equation}
We next fix $\Omega_1,\Omega_2$ and apply Proposition \ref{rejgiotjgiotjio} to each $Q\subset X$. We analyze the only nontrivial scenario, when \eqref{cccc1} holds for each $Q\subset X$. As before, we may assume that the resulting family $\{\omega_i\}$ is the same for each $Q$. We sum \eqref{cccc1} over $Q\subset X$
and use Cauchy--Schwarz
$$\int_{X}|f_{\Omega_1}f_{\Omega_2}|^2\lesssim K_1^{O(1)}(\sum_{Q\subset X}(\sum_{\omega_1\subset \Omega_1}\|g_{\omega_1}\|_{L^4(Q)}^2)^2)^{1/2}(\sum_{Q\subset X}(\sum_{\omega_2\subset \Omega_2}\|g_{\omega_2}\|_{L^4(Q)}^2)^2)^{1/2}.$$
We next apply Proposition \ref{fjgjirtgihiythiuij09u} to each of the terms
$$\int_{X}|f_{\Omega_1}f_{\Omega_2}|^2\lesssim (M_1M_2)^{1/2}(\sum_{T\in\ZT_{\Omega_1}}\big\|f_{T}\big\|_4^4)^{1/2}(\sum_{T\in\ZT_{\Omega_2}}\big\|f_{T}\big\|_4^4)^{1/2}\;\times$$
$$K_1^{O(1)}\left((\log R)^{O(1)}\sup_{r'\lesssim K_3^{-1}}(\frac{1}{r'})^{200\frac{\log{K_1}}{\log K_2}}C(R(r')^2/K_1^2)+R^\e\right).$$
The theorem follows by combining this with \eqref{f54kgjrihyihiyhiytuh}. \qedhere

\end{proof}

\begin{proof}[Proof of Theorem \ref{refined-dec-thm}]

The proof of Theorem \ref{refined-dec-thm} as a corollary of Theorem \ref{jregiurtiguiuhyuh8}
is standard. 
We assume $C(R)\sim R^\alpha$, and prove that $\alpha\le 2\e$ for all $\e>0$. 
The choice of $K_1,K_2,K_3$ should be in such a way that prevents the first term in \eqref{fiojugurtuhgiyt9hi0967i} to dominate when $\alpha=2\e$. That means, we need 
$$R^{\alpha}\gg R^\alpha (K_1K_2)^{O(1)}\frac1{K_3^{2\alpha-200\e^2}},\text{ with }\alpha=2\e.$$
This means, we need 
\begin{equation}
\label{ejiorjfiougioturgiou}	
K_3\ge (K_1K_2)^{1/\e}.
\end{equation}
This justifies our initial choice for $K_3$. 

Since we now know that the second
term in \eqref{fiojugurtuhgiyt9hi0967i} dominates, we are left with 
\[C(R)\lesssim_\e (K_1K_2)^{O(1)}R^\e\lesssim R^{2\e}. \qedhere\]

\end{proof}
	
\bigskip

\section{Restriction estimates}

We start by pointing out a few key differences in our notation here, compared to the earlier sections. 
Throughout this section, $f$ will be a function of two (rather than three) variables, that we denote by $(\xi_1,\xi_2)$ (rather than $(\xi,\eta)$). Given a rectangle $\tau\subset [-1,1]^2$, the notation $f_\tau$ will now be reserved to denote $f\Id_\tau$.

Standard arguments reduce Conjecture \ref{restriction-conj} for $n=3$, $S=\ZH_{[-1,1]^2}$ to the following version.
\begin{conjecture}
\label{restriction-conj-local}
Define the extension operator
$$
    E f(x_1,x_2,x_3):=\int_{[-1,1]^2}e^{i(x_1\xi_1+x_2\xi_2+x_3\xi_1\xi_2)}f(\xi_1, \xi_2)d\xi_1d\xi_2.
$$
Then for $p>3$, the following is true:
For all $\e>0$, there exists $C_\e>0$ such that for all $R\geq 1$,
\begin{equation}
\label{restriction-esti-local}
    \|Ef\|_{L^p(B_R)}^p\leq C_\e R^\e\|f\|_p^p.
\end{equation}
\end{conjecture}

Thus, to prove Theorem \ref{restriction-thm}, it suffices to prove the following result.
\begin{theorem}
\label{restriction-thm-local}
Inequality \eqref{restriction-esti-local} is true when $p= 22/7$.
\end{theorem}

\medskip

\subsection{Wave packet decomposition and incidence geometry}
\label{ssWP}
We will construct a wave packet decomposition and state some of its key properties for later use.
The wave packet decomposition is quite standard nowadays. 
See, for example, \cite{Demeter-book}.

\smallskip

Given the $\e$ in Conjecture \ref{restriction-conj-local}, we fix the tiny constant 
\begin{equation}
\label{fiourf8ur8gut8gu8tu}
\e_0=\e^{1000}.
\end{equation}
In the frequency space, let $\Theta$ be a finite-overlapping cover of $[-1,1]^2$ by $R^{-1/2}$-balls $\theta$, and let $\{\vp_\theta\}_{\theta\in \Theta}$ be a smooth partition of unity so that $\supp(\vp_\theta)\subset 2\theta$ and $\sum_{\theta\in\Theta} \vp_\theta=1$ on $[-1,1]^2$. For $f:[-1,1]^2\to\C$ we abuse earlier notation and write
$$f_\theta=f\vp_\theta.$$

In the physical space, let $\cv$ be a finite-overlapping partition of $\ZR^{2}$ by $R^{1/2}$-balls, and let $\{\psi_v\}_{v\in \cv}$ be a smooth partition of unity of $\ZR^{2}$ so that $\wh\psi_v$ is concentrated near $v$, $\supp(\wh\psi_v)\subset B^{2}(0,R^{-1/2})$ and $\sum_{v\in\cv}\psi_v=1$ in $\ZR^{2}$.

\smallskip

The above frequency-space partition gives the wave packet decomposition for any function $f$ supported on $[-1,1]^2$
\begin{equation}
	\nonumber
	f=\sum_{\theta\in \Theta}\sum_{v\in\cv}(f\vp_\theta)\ast\hat\psi_v=:\sum_{(\theta,v)\in\Theta\times\cv}f_{\theta,v}.
\end{equation}
For $x\in\ZR^3$, write $x=(\bar x, x_3)$. 
Let $\Phi(\xi)=\xi_1\xi_2$.
For each $\theta\in \Theta$ and each $v\in\cv$, let $T_{\theta,v}=\{(\bar x,x_3)\in B_R:|\bar x-c_v+x_3\nabla\Phi(c_\theta)|\leq R^{1/2+\e_0} \}$ be a tube of dimensions $R^{1/2+\e_0}\times R^{1/2+\e_0}\times R$, where $c_\theta,c_v$ are the centers of $\theta,v$ respectively. 
Denote by $V(\theta)$ the vector $(1, \nabla\Phi(c_\theta))$.
Let $\bar\ZT(\theta)=\{T_{\theta,v}:v\in\cv\text{ and }T_{\theta,v}\cap B_R\not=\varnothing\}$ be a family of $R$-tubes with direction $V(\theta)$, and let $\bar\ZT=\bigcup_\theta\bar\ZT(\theta)$. 
If $T=T_{\theta,v}$, we write 
\begin{equation}
\label{jiojgiotgiutu8hu8y}
f_T=f_{\theta,v}, \;\theta=\theta_T.
\end{equation}

\smallskip

The next lemma is standard.

\begin{lemma}
	\label{wpt}
	The wave packet decomposition satisfies the following properties.
	\begin{enumerate}\item $Ef=\sum_{T\in\bar\ZT}Ef_{T}$.
		\item $|Ef_{T}(x)|\lesssim R^{-1000}$ when $x\in B_R\setminus T$.
		\item $\supp f_{T}\subset 3\theta$ when $T$ has direction $V(\theta)$.
		\item $\{V(\theta)\}_{\theta\in\Theta}$ are $\gtrsim R^{-1/2}$-separated.
		\item $\bar\ZT(\theta)$ is $R^{O(\e_0)}$-overlapping.
		\item $\|Ef_T\|_{L^p(w_{B_R})}\lesssim R^{2(\frac{1}{p}-\frac{1}{2})} \|Ef_T\|_{L^2(w_{B_R})}$ for all $T\in\bar\ZT $, where $w_{B_R}$ is a weight that is $\sim1$ on $B_R$ and decreases rapidly outside $B_R$. 
	\end{enumerate}    
\end{lemma}

\begin{remark}
	\rm
	
	The fact that $\{V(\theta)\}_{\theta\in\Theta}$ are $\gtrsim R^{-1/2}$-separated is crucial, as it allows us to use Proposition \ref{kakeya-prop} to handle the incidence geometry of wave packets.
\end{remark}

\medskip

The next two results were proved in \cite[Lemma 4.5]{Wang-Wu} and \cite[Proposition 3.2]{Wang-Wu}, respectively.

\begin{lemma}
	\label{lem: l2}
	Let $X$ be a union of $R^{1/2}$-balls, and let $f=\sum_{T\in\ZT}f_T$ be a sum of wave packets.
	Suppose for each $T\in \ZT$ there is a shading $Y(T)\subset T$ by $R^{1/2}$-balls in $X$ such that the number of $R^{1/2}$-balls intersecting $Y(T)$ is $\lesssim \lambda R^{1/2}$.
	Then
	\begin{equation}
		\nonumber
		\int_{X}\big|\sum_{T\in\ZT}Ef_{T}\Id_{Y(T)}\big|^2\lesssim (\la R)\|f\|_2^2.
	\end{equation}
\end{lemma}

\begin{proposition}
	\label{kakeya-prop}
	Let $\de\in(0,1)$.
	Let $(L,Y)_\de$ be a collection of $\de$-separated lines together with an $(\e_1, \e_2)$-two-ends, $\la$-dense shading. 
	Define $E_L=\bigcup_{\ell\in L}Y(\ell)$.
	Suppose that every $\de$-ball on $\ZS^{2}$ contains $\leq m$ points from the direction set $\{V(\ell):\ell\in L\}$, where $V(\ell)$ is the direction of $\ell$.
	Take $\mu=\de^{-2\e_1}m\la^{-3/4}\de^{-1/2}$. Then there exists a set $E_\mu\subset E_L$ such that $\# L(x)\lessapprox \mu$ for all $x\in E_\mu$, and
	\begin{equation}
		\nonumber
		|E_L\setminus E_\mu|\leq \de^{\e_1}|E_L|.
	\end{equation}
\end{proposition}

\medskip 

\subsection{The broad-narrow reduction}

What follows is a somewhat standard broad-narrow argument. 
The broad function considered here is slightly different from the one introduced in \cite{Guth-R3}. It needs to incorporate the more severe notion of transversality for $\ZH$, as introduced in 
Definition \ref{transversality}.

Assume $K\ge 1$ is dyadic. Let us denote by $\cC_K$ the collection of all dyadic $1/K$-squares in $[-1,1]^2$.
\begin{definition}
	\label{def-broad}
	Let $K\geq A\geq1$.
	We say that a collection $\Tau=\{\tau\}\subset \cC_K$   is {\bf $A$-broad} if
	\begin{enumerate}
		\item $\#\Tau\geq A$.
		\item For $j=1,2$, the $\xi_j$ coordinates $\{c(\tau)_j:\;\tau\in\Tau\}$ of the centers  are $2K^{-1}$-separated.
	\end{enumerate}
\end{definition}
In the proof of Theorem \ref{GMOt} we have referred to the second requirement as $K^{-1}$-transversality.
We note that when $K\sim 1$, this is essentially the same as the concept introduced in  Definition \ref{transversality}.

\begin{definition}
	\label{broad-norm}
	Let $K\geq A\geq1$. Consider a collection $\{F^\tau\}_{\tau\in\cC_K}$ of functions $F^\tau:\R^3\to\C$.
	For any $x\in\ZR^3$, we define the broad function $\Br_A \{F^\tau\}(x)$ as
	\begin{equation}
		\nonumber
		\Br_A \{F^\tau\}(x)=\max_{\cT:\text{ $\cT$ is $A$-broad}}\min_{\tau\in\cT}\;|F^\tau(x)|.
	\end{equation}
\end{definition}
We note that $A\mapsto \Br_A \{F^\tau\}(x)$ is non-increasing. The following two observations are immediate.
\begin{lemma}
	If $A=A_1+A_2+\ldots+A_N$ and $F^{\tau}=F^{\tau}_1+F^{\tau}_2+\ldots+F^{\tau}_N$, then
	\begin{equation}
		\label{broad-triangle}
		\Br_A \{F^\tau\}(x)\leq \Br_{A_1} \{F^{\tau}_1\}(x)+\Br_{A_2}\{F^{\tau}_2\}(x)+\ldots+\Br_{A_N}\{F^{\tau}_N\}(x).
	\end{equation}
\end{lemma}
\begin{lemma}
	If $A\geq 2$, then
	\begin{equation}
		\label{broad-bilinear}
		\Br_A \{F^{\tau}\}(x)\leq \max_{\tau_1,\tau_2:\;K^{-1}-\text{transverse}}|F^{\tau_1}(x)F^{\tau_2}(x)|^{1/2}.
	\end{equation}
\end{lemma}
Most of our applications will concern the case when $F^{\tau}=\sum_{T\in\bar{\ZT}\atop{\theta_T\subset\tau}}Ef_T$ for some $f$. We note that the latter equals $Ef_\tau$, where $f_\tau=f1_{\tau}$. In this case, we simply write $\Br_AEf(x)$ for $\Br_A \{F^{\tau}\}$.

The broad norm will be needed in our later arguments, as, unlike the geometric averages in \eqref{broad-bilinear}, it satisfies the (quasi)-triangle inequality \eqref{broad-triangle}. In all our applications, $N$ will be $O((\log R)^{O(1)})$ and the values of $A_i$ will be only logarithmically smaller than $A$.

We now prove the main result in this subsection. The three terms in \eqref{broad-narrow} from below correspond to the terms in the three cases from the proof of Theorem \ref{GMOt}. 

\begin{proposition}[Broad-narrow reduction]
	\label{broad-narrow-lem}
	Let $K\gg 1$.
	Let $\{S_1\}$ be a family of horizontal $1\times K^{-1}$-rectangles, and let $\{S_2\}$ be a family of vertical $K^{-1}\times 1$-rectangles, such that both $\{S_1\}$ and $\{S_2\}$ partition $[-1,1]^2$.
	Then for all $x\in\ZR^3$ and $\e>0$ we have 
	\begin{align}
		\label{broad-narrow}
		|Ef(x)|\lesssim_\e K^{5\e}\max_{\tau\in\cC_K}|Ef_{\tau}(x)|+ K^{2\e}\max_{S_j}|Ef_{S_j}(x)|+K^{3}\cdot\Br_{K^\e} Ef(x).
	\end{align}
\end{proposition}

\begin{proof}
	
	If there exists a square $\tau$ such that $|Ef_\tau(x)|\geq K^{-5\e}|Ef(x)|$, then the first term in \eqref{broad-narrow} dominates $|Ef(x)|$.
	Otherwise, $|Ef_\tau(x)|\leq K^{-5\e}|Ef(x)|$ for all $\tau$.
	
	For $j=1,2$, if there exists an $S_j$ such that $|Ef_{S_j}(x)|\leq K^{-2\e} |E(x)|$, then $|Ef(x)|$ is dominated by the second term of \eqref{broad-narrow}.
	Otherwise, $|Ef_{S_j}(x)|\leq K^{-2\e} |Ef(x)|$ for all $S_j$.
	Let $\cT$ be the family of $K^{-1}$-squares that $|Ef_\tau(x)|\geq K^{-3}|Ef(x)|$, so 
	\begin{equation}
		\nonumber
		\Big|\sum_{\tau\in\cT}Ef_\tau(x)\Big|\geq(1/2)|Ef(x)|.
	\end{equation}
	
	We next prove that $\cT$ cannot  be covered by a family consisting of horizontal strips $\cs_1$ and vertical strips $\cs_2$, such that $\#\cs_1, \#\cs_2\leq 3K^\e$. Indeed, assume for contradiction that such a family exists. We write
	\begin{equation}
		\nonumber
		\sum_{\tau\in\cT}Ef_\tau(x)=\sum_{S_1\in\cs_1}Ef_{S_1}(x)+\sum_{S_2\in\cs_2}Ef_{S_2}(x)-\sum_{\substack{\tau\in\cT:\;\tau\subset S_1\cap S_2\\ \text{for some $S_j\in\cs_j$}}}Ef_\tau(x).
	\end{equation}
	As a result, we have
	\begin{equation}
		\nonumber
		\sum_{S_1\in\cs_1}|Ef_{S_1}(x)|+\sum_{S_2\in\cs_2}|Ef_{S_2}(x)|+\sum_{\substack{\tau\in\cT:\tau\subset S_1\cap S_2\\ \text{for some $S_j\in\cs_j$}}}|Ef_\tau(x)|\geq(1/2)|Ef(x)|.
	\end{equation}
	Since $\#\cs_j\leq 3K^\e$, we must have $\#\{\tau\in\cT:\tau\subset S_1\cap S_2 \text{ for some $S_j\in\cs_j$}\}\le 9K^{2\e}$.
	Also, recall that $\max_{S_j}|Ef_{S_j}(x)|\leq K^{-2\e}|Ef(x)|$, $\max_{\tau}|Ef_{\tau}(x)|\leq K^{-5\e}|Ef(x)|$.
	This contradicts the inequality above.

	Finally, we claim that there exists $\cT(x)\subset\cT$ such that $\cT(x)$ is $K^\e$-broad. Thus, the third term in \eqref{broad-narrow} dominates $|Ef(x)|$.
	We construct $\cT(x)$ inductively. Start with any $\tau_1\in\cT$. Pick $\tau_2\in\cT$ not contained inside any of the three horizontal or the three vertical strips either containing or adjacent to $\tau_1$. Assuming $\tau_1,\ldots,\tau_{n-1}$ have been constructed, pick $\tau_n\in \cT$ not contained in any of the strips containing or adjacent to any of the $\tau_1,\ldots,\tau_{n-1}$. There are at most $3(n-1)$ such horizontal or vertical strips, so this process may continue at least as long as $n\le K^{\e}$. The resulting collection is easily seen to be $K^\e$-broad.
	\qedhere

\end{proof}

\smallskip

\subsection{The estimate for the broad norm}
We start with a combinatorial lemma that will be used repeatedly in this section.
\begin{lemma}[Pigeonholing]
\label{lcomb}Consider a finite collection of numbers $I_Q,\;Q\in\mathcal Q$, with $I_Q\in [L,2L]$. Assume there is a finite set $\Lambda$ and numbers $A\le I_{Q,\lambda}\le B$ such that for each $Q\in\mathcal Q$
$$I_Q\le C\sum_{\lambda\in\Lambda}I_{Q,\lambda}.$$
Then there are $\lambda\in\Lambda$, $L'$, and $\mathcal Q''\subset \mathcal Q$ such that $I_{Q,\lambda}\in [L',2L']$ for each $Q\in\mathcal Q''$,
$$ (\log B/A)^{-1}(\#\Lambda)^{-1}\#\mathcal Q\le\#\mathcal Q'' $$ and
$$(C\log B/A)^{-1}(\#\Lambda)^{-2}\sum_{Q\in \mathcal Q}I_Q\lesssim \sum_{Q\in\mathcal Q''}I_{Q,\lambda}.$$
\end{lemma}
\begin{proof}
For each $Q$, pick $\lambda_Q\in\Lambda$ such that $I_{Q,\lambda_Q}\ge (C\# \Lambda)^{-1}I_Q.$ Then pick a collection $\mathcal Q'\subset\mathcal Q$ such that $\#\mathcal Q'\ge (\#\Lambda)^{-1}\#\mathcal Q$ and $\lambda_Q$ is the same for $Q\in\mathcal Q'$. Call $\lambda$ the common value. 

Finally, pick $\mathcal Q''\subset \mathcal Q'$ such that $\#\mathcal Q''\ge (\log B/A)^{-1}\#\mathcal Q'$, and moreover, there is $L'$ such that  $I_{Q,\lambda}\in [L',2L']$ for each $Q\in\mathcal Q''$.
\qedhere
\end{proof}

\medskip

Given $f:[-1,1]^2\to \C$, we write
$$f=\sum_{\theta\in\Theta}f_\theta.$$
We prove our main result about the broad norm. 
\begin{proposition}
	\label{broad-prop}
	Assume $\e\ll 10^{-3}$ is small enough.
	For $R\gg 1$, let $K=R^{\e^{10}}$.
	Then there exists $C_\e>0$ such that for all $R\geq1$,
	\begin{equation}
		\label{mixed-norm}
		\int_{B_R}|\Br_{A} Ef|^p\leq C_\e R^{2\e}\|f\|_2^2\;\sup_{\theta:R^{-1/2}\text{-square}}\|f_\theta\|_{L^2_{avg}(\theta)}^{p-2}
	\end{equation}
	for $p=22/7$ and $A\geq R^{\e^{20}}$.
	Here $\|f_\theta\|_{L^2_{avg}(\theta)}$ is defined as
	\begin{equation}
		\|f_\theta\|_{L^2_{avg}(\theta)}^2:=|\theta|^{-1}\|f_\theta\|_2^2.
	\end{equation}
\end{proposition}
\begin{proof}
	Throughout this argument, we let $p=22/7$. Fix $\e$.
	We use induction on  $r$ to prove that for all $r\in[R^{\e^2},R]$ and $A\geq r^{\e^{20}}$,
	\begin{equation}
		\label{mixed-norm-2}
		\int_{B_r}|\Br_{A} Ef|^p\leq C_\e R^\e r^\e\|f\|_2^2\sup_{\theta:r^{-1/2}\text{-square}}\|f_\theta\|_{L^2_{avg}(\theta)}^{p-2}.
	\end{equation}
	The base case is when $r=R^{\e^2}$, which is trivial via the use of elementary inequalities,  as $R^{\e}=r^{\e^{-1}}\geq r^{100}$. 
	We will see that the number $n$ of steps in this iteration is $\sim \log\e/\log(1-\e)=O_\e(1)$. Indeed, the sequence of radii is $$R^{\e^2},\;R^{\e^2/(1-\e^2)},\ldots, R^{\e^2/(1-\e^2)^n}\sim R.$$
	
	\smallskip

    Assume \eqref{mixed-norm-2} holds for $r=R^{\e^2/(1-\e^2)^{m-1}}$, $m\ge 1$. Fix $r=R^{\e^2/(1-\e^2)^m}$ and fix $B_r$. Partition the $r$-tubes $\bar\ZT=\ZT\sqcup\ZT_{small}$, where 
    $\ZT_{small}=\{T\in\bar\ZT:\|f_T\|_2\leq r^{-100}\|f\|_2\}$.
	An easy computation shows that $ \int_{B_r}|\Br_{A} E(\sum_{T\in\ZT_{small}}f_T)|^p\lesssim r^{-10}\|f\|_2^p$, which trivially yields \eqref{mixed-norm}. It remains to estimate $\int_{B_r}|\Br_{A} Ef'|^p,$
    where $f'=\sum_{T\in\ZT}f_T$.
    Next, partition $\ZT=\bigsqcup_{\ga,m}\ZT_{\ga,m}$, where $\ga,m\in[r^{-100},r^{10}]$ are dyadic numbers, such that 
	\begin{itemize}
		\item For all $T\in\ZT_{\ga,m}$, $\|f_T\|_2\sim \ga\|f\|_2$\;.
		\item For all $\theta$, either $\ZT_{\ga,m}(\theta)=\varnothing$, or $\#\ZT_{\ga,m}(\theta)\sim m$.
	\end{itemize}
	Since there are $O((\log r)^2)$ possible pairs of dyadic numbers $(\ga,m)$, by the triangle inequality \eqref{broad-triangle}, there exists a pair  $(\ga,m)$ and  $A_g\gtrapprox A$ such that, writing $g=\sum_{T\in\ZT_{\ga,m}}f_T$, we have 
	\begin{equation}
		\label{jfdjiofgjbiogjbigji}
		\int_{B_r}|\Br_{A} Ef'|^p\lessapprox \int_{B_r}|\Br_{A_g} Eg|^p.
	\end{equation}
	To ease notation, we let $\ZT_g=\ZT_{\ga,m}$.
    By dyadic pigeonholing, there exists a union $X$ of $r^{1/2}$-balls $Q$ such that 
	\begin{itemize}
		\item The values  $\int_{Q}|\Br_{A_g} Eg|^p$ are about the same for all $Q\subset X$.
		\item We have
		\begin{equation}
			\label{reduction-1}
			\int_{B_r}|\Br_{A_g} Eg|^p\lessapprox\int_{X}|\Br_{A_g} Eg|^p.
		\end{equation}
	\end{itemize}

	\medskip
	
	\noindent{\bf Step 1: Two-ends reduction.}
	
	Partition each $r$-tube $T\in\ZT_g$ into  tube segments $J$ of length $r^{1-\e^2}$. Let $\cj(T)$ be those segments that intersect $X$. 
	Then, partition  $\cj(T)=\bigcup_\la\cj_\la(T)$, $\la\in\Lambda$, where $\Lambda$ denotes the dyadic numbers in $[r^{-1/2},r^{-\e^2}]$, and $|J\cap X|\sim \la |T|$ for any $J\in\cj_\la(T)$. 
	Thus, 
	\begin{equation}
		\nonumber
		Eg=\sum_{\la}\sum_{T\in\ZT_g}Ef_{T}\sum_{J\in\cj_\la(T)}\Id_J.    
	\end{equation}
    Write $F_\lambda^\tau=\sum_{T\in\ZT_g\atop{\theta_T\subset\tau}}Ef_{T}\sum_{J\in\cj_\la(T)}\Id_J$, $\tau\in\cC_K$. Note that $Eg_\tau=\sum_{\lambda}F_\lambda^{\tau}$. The triangle  inequality \eqref{broad-triangle} together with the triangle inequality in $L^p$ followed by H\"older imply that, for some $A_1\gtrapprox A_g$, we have for each $Q\subset X$
    $$\int_{Q}|\Br_{A_g} Eg|^p\le \#(\Lambda)^{p-1}\sum_{\lambda}\int_{Q}|\Br_{A_1} \{F_\lambda^{\tau}\}|^p.$$

    We may assume all nonzero terms $I_{Q,\lambda}:=\int_{Q}|\Br_{A_1} \{F_\lambda^{\tau}\}|^p$ are in the interval $[r^{-100}(\gamma\|f\|_2)^p,r^{100}(\gamma\|f\|_2)^p]$.    
	Since $I_Q:=\int_{Q}|\Br_{A_g} Eg|^p$ are about the same for all $Q\subset X$, and since $\#\Lambda\lessapprox 1$, by Lemma \ref{lcomb} there is a $\la\in\Lambda$ and a set of $r^{1/2}$-balls $X_1\subset X$ such that 
	\begin{itemize}
		\item $|X_1|\gtrapprox|X|$.
		\item For each $r^{1/2}$-ball $Q\subset X_1$, 
        $\int_{Q}|\Br_{A_1} \{F_\lambda^{\tau}\}|^p
        $ has about the same value.
		\item We have
		\begin{equation}			
			\int_{X}|\Br_{A_g} Eg|^p\lessapprox\int_{X_1}|\Br_{A_1} \{F_\lambda^{\tau}\}|^p.
		\end{equation}
	\end{itemize}
	Consider the partition $\ZT_g=\bigcup_\be\ZT_{\be}$, where $\be\in[1,r^{\e^2}]$ is a dyadic number and $\#\cj_\la(T)\sim\be$ for all $T\in\ZT_{\be}$. 
	As a result, 
	\begin{equation}
		\nonumber
		\sum_{T\in\ZT_g}\sum_{J\in\cj_\la(T)}Ef_{T}\Id_J=\sum_\be\sum_{T\in\ZT_{\be}}\sum_{J\in\cj_\la(T)}Ef_{T}\Id_J.    
	\end{equation}Write $F_{\lambda,\beta}^\tau=\sum_{T\in\ZT_\beta\atop{\theta_T\subset\tau}}Ef_{T}\sum_{J\in\cj_\la(T)}\Id_J$, $\tau\in\cC_K$. Note that $F_\lambda^\tau=\sum_{\beta}F_{\lambda,\beta}^{\tau}$.
	Reasoning as in the previous step, using the triangle  inequality \eqref{broad-triangle} and  Lemma \ref{lcomb}, 
     we find  $\be$, $A_2\gtrapprox A_1$ and a set of $r^{1/2}$-balls $X_2$ such that 
	\begin{itemize}
		\item $|X_2|\gtrapprox|X_1|$.
		\item For each $r^{1/2}$-ball $Q\subset X_2$, $\int_{Q}|\Br_{A_2} \{F_{\lambda,\beta}^{\tau}\}|^p$ has about the same value.
		\item We have
		\begin{equation}
			\label{reduction-2}
			\int_{X_1}|\Br_{A_1} \{F_\lambda^{\tau}\}|^p\lessapprox\int_{X_2}|\Br_{A_2} \{F_{\lambda,\beta}^{\tau}\}|^p .
		\end{equation}
	\end{itemize}
	
	\smallskip
    
    It remains to analyze the last integral. We will distinguish two cases.
	Let $\{B_k\}$ be a finitely overlapping family of $r^{1-\e^2}$-balls that cover $B_r$. 
	\medskip
	
	\noindent{\bf Step 2: The non-two-ends scenario.} Assume $\be\leq r^{\e^4}$.
	
	For each $B_k$, define 
	\begin{equation}
		\nonumber
		g_{k}=\sum_{\substack{T\in\ZT_{\be} \text{ such that}\\ \exists J\in\cj_\la(T),\, J\cap B_k\not=\varnothing}} f_{T}.
	\end{equation}
    Note that on each $B_k$, by Lemma \ref{wpt}
    $$\big|\sum_{T\in\ZT_{\be}}Ef_T\sum_{J\in\cj_{\la}(T)}\Id_J\big|\sim |Eg_k|.$$
	Thus, we have
	\begin{align}
		\label{related}
		\int_{X_2\cap B_k}\big|\Br_{A_2}\{F_{\lambda,\beta}^\tau\}\big|^p&\sim\int_{X_2\cap B_k}\big|\Br_{A_2} Eg_k\big|^p.
	\end{align}
	Note that for each $T$, there are $\lesssim r^{\e^{4}}$ many $B_k$ such that $\exists J\in\cj_\la(T), J\cap B_k\not=\varnothing$.
	As a consequence, 
	\begin{equation}
		\label{l2-related}
		\sum_{k}\|g_k\|_2^2\lesssim r^{\e^{4}}\|g\|_2^2\lesssim r^{\e^4}\|f\|_2^2. 
	\end{equation}
	Since $A_2\gtrapprox A\geq r^{\e^{20}}$, we have (for $\e$ small enough) $A_2\geq r^{(1-\e^2)\e^{20}}$.
	Apply \eqref{mixed-norm-2} as an induction hypothesis on each $r^{1-\e^2}=R^{\e^2/(1-\e^2)^{m-1}}$-ball $B_k$ to get
	\begin{equation}
		\label{vpofigtug8utr8gurtg98u}
		\|\Br_{A_2} Eg_k\|_{L^p(B_k)}^p\leq C_\e R^\e r^{(1-\e^2)\e}\|g_k\|_2^2\;\sup_{\om}\|g_{k,\om}\|_{L_{avg}^2(\om)}^{p-2},
	\end{equation}
	where the sup is over  $ r^{(\e^2-1)/2}$-squares $\om$.
	By $L^2$-orthogonality,  
    \begin{equation}
    \label{rpojirutgut8}
    \sup_{\om}\|g_{k,\om}\|_{L^2_{avg}(\om)}^{p-2}\lesssim\sup_{\theta}\|f_{\theta}\|_{L^2_{avg}(\theta)}^{p-2}.
    \end{equation}
	Summing up over $B_k$, using \eqref{jfdjiofgjbiogjbigji}-\eqref{rpojirutgut8},  we find
	\begin{align}
		\nonumber
		\int_{B_r}|\Br_{A} Ef'|^p&\lessapprox \, \sum_{k}C_\e R^\e r^{(1-\e^2)\e} \|g_k\|_2^2\sup_{\om}\|g_{k,\om}\|_{L^2_{avg}(\om)}^{p-2}\\ \nonumber
		&\lesssim  r^{-\e^3+\e^4}C_\e R^\e r^{\e}\|f\|_2^2\sup_{\theta}\|f_{\theta}\|_{L^2_{avg}(\theta)}^{p-2}.
	\end{align}
	This proves \eqref{mixed-norm-2}. Note that  we  have not used yet either the information (gained via pigeonholing) regarding the subsets $X_i$ or the constant property relative to $Q$. These will be used in the next step, more precisely, in the derivation of \eqref{bdbhdbvhhvbfdhvbhfb}.  
	
	\medskip
	
	\noindent{\bf Step 3: The two-ends case.} Assume $\be\ge r^{\e^4}$.
    
	For each $T\in\ZT_{\be}$, consider the shading $Y(T)=\bigcup_{J\in\cj_\la(T)}(J\cap X)$.
	Then $Y$ is a rescaled $(\e^2, \e^{4})$-two-ends, $\la\be$-dense shading.

	\smallskip

	Define
	\begin{equation}
		\nonumber
		\mu=r^{2\e^{2}}m(\la\be)^{-3/4}r^{1/4}.
	\end{equation}
	We  apply Proposition \ref{kakeya-prop} to the $r^{-1}$-dilate of $(\ZT_\be, Y)$ (with $\delta=r^{-1/2}$) to obtain a set $X_3\subset X$ with
	 $|X\setminus X_3|\leq r^{-\e^2}|X|$ and ( recall $\e_0$ from \eqref{fiourf8ur8gut8gu8tu})
    \begin{equation}
		\label{multi-R-half}
		\sup_{Q\subset X_3}\#\{T\in\ZT_\be:Y(T)\cap Q\not=\varnothing\}\lessapprox r^{O(\e_0)}\mu.
	\end{equation}
	
	Since $|X_2|\gtrapprox |X|$, we know that  $|X_2\setminus X_3|\lessapprox r^{-\e^2}|X_2|$. 
	Denote by $X_4=X_2\cap X_3$, so we have $|X_4|\gtrapprox|X_2|$ and $X_4\subset X_2$. 
	Recall that $\int_{Q}|\Br_{A_2} \{F_{\lambda,\beta}^{\tau}\}|^p$
     are about the same for $Q\subset X_2$.
	Thus,  we have 
	\begin{equation}
		\label{bdbhdbvhhvbfdhvbhfb}\int_{X_2}|\Br_{A_2} \{F_{\lambda,\beta}^{\tau}\}|^p
		\lessapprox \int_{X_4}|\Br_{A_2} \{F_{\lambda,\beta}^{\tau}\}|^p.
	\end{equation}
    The change of the domain of integration from $X_2$ to $X_4$ is crucial, as it will give us access to the incidence estimate \eqref{multi-R-half}.
    
	Assuming $\e$ is small enough, we have $A_2\ge 2$, as $A_2\gtrapprox A\ge r^{\e^{20}}$.
	We invoke \eqref{broad-bilinear} and pigeonholing to find two $K^{-1}$-transverse $\tau_1,\tau_2\in\cC_K$ so that, denoting $\ZT_{\be}[\tau_j]=\bigcup_{\theta\subset\tau_j}\ZT_{\be}(\theta)$, we have
	\begin{align}
		\nonumber
        \int_{X_4}|\Br_{A_2} \{F_{\lambda,\beta}^{\tau}\}|^p
		\lesssim K^{O(1)}\int_{ X_4}\prod_{j=1,2}\Big|\sum_{T\in\ZT_{\be}[\tau_j]}\sum_{J\in\cj_{\la}(T)}Ef_{T}(x)\Id_J\Big|^{p/2}.
	\end{align}
	Recall that $\{B_k\}$ is a partition of $B_r$ into $r^{1-\e^2}$-balls. 
	For each $B_k$, let $$\ZT_{\be, k}[\tau_j]=\{T\in\ZT_{\be}[\tau_j]: \exists \;J\in\cj_{\la}(T),\; J\cap B_k\not=\varnothing\}.$$
	Therefore, using earlier inequalities we find
	\begin{align}
		\label{X-4}
		\int_{B_r}|\Br_{A} Ef'|^p&\lessapprox K^{O(1)} \sum_k\int_{X_4\cap B_k}\prod_{j=1,2}\Big|\sum_{T\in\ZT_{\be}[\tau_j]}\sum_{J\in\cj_{\la}(T)}Ef_{T}(x)\Id_J\Big|^{p/2}\\ \nonumber
		&\sim K^{O(1)}\sum_k\int_{X_4\cap B_k}\prod_{j=1,2}\Big|\sum_{T\in\ZT_{\be, k}[\tau_j]}Ef_{T}\Big|^{p/2}\\ 
        \nonumber
		&\lesssim  K^{O(1)} r^{10\e^2}\max_k\int_{X_4\cap B_k}\prod_{j=1,2}\Big|\sum_{T\in\ZT_{\be, k}[\tau_j]}Ef_{T}\Big|^{p/2}.
	\end{align}
    We derive two estimates for the right-hand side. 
    
	Notice that for each $Q\subset B_k$,
	\begin{equation}
		\nonumber
		\{T\in\ZT_{\be}:Y(T)\cap Q\not=\varnothing\}=\{T\in\ZT_{\be,k}:T\cap Q\not=\varnothing\}.
	\end{equation}
    When combined with \eqref{multi-R-half}, this shows that when $Q\subset X_4\cap B_k$
    \begin{equation}
    \label{multi-R-half4}
    \#\{T\in\ZT_{\be, k}:T\cap Q\not=\varnothing\}\lessapprox r^{O(\e_0)}\mu.
    \end{equation}
	At this point, we invoke Theorem \ref{refined-dec-thm} at scale $r$, using the set $X_4\cap B_k\subset B_r$ and  the bound \eqref{multi-R-half4}  to have
	\begin{equation}
		\label{jkijurgu98tug8tug89t}
		\int_{X_4\cap B_k}\prod_{j=1,2}\Big|\sum_{T\in\ZT_{\be, k}[\tau_j]}Ef_{T}\Big|^2\lessapprox K^{O(1)} r^{O(\e_0)} \mu\sum_{T\in\ZT_g}\big\|Ef_T\big\|_{L^4(w_{B_r})}^4.
	\end{equation}
	
	Recall that $\|f_T\|_{2}$ have comparable magnitude for all $T\in\ZT_g$, and that $\#\ZT_g(\theta)\sim m$ for all $\theta$ such that $\ZT_g(\theta)\not=\varnothing$.
	Thus, for each $\theta'$ we have
	\begin{align}
		\nonumber
		\sum_{T\in\ZT_g(\theta')}&\big\|Ef_T\big\|_{L^4(w_{B_r})}^4\lesssim r^{-2}\sum_{T\in\ZT_g(\theta')}\big\|Ef_T\big\|_{L^2(w_{B_r})}^4\lesssim\sum_{T\in\ZT_g(\theta')}\|f_T\|_2^4\\
		\label{after-decoupling}
		&\lesssim m^{-1} \Big(\sum_{T\in\ZT_g(\theta')}\big\|f_T\big\|_{2}^2\Big)^2\lesssim  (mr)^{-1}\|f_{\theta'}\|_2^2\;\sup_\theta\|f_\theta\|_{L^2_{avg}(\theta)}^2.
	\end{align} 
	We recall that $O(\e_0)\leq\e^2$, $K=R^{\e^{10}}\leq r^{\e^8}$, and $\be\ge 1$.
	Thus, summing up over all $\theta'$ in \eqref{after-decoupling} and plugging it back into \eqref{X-4},\eqref{jkijurgu98tug8tug89t}, we have
	\begin{align}
		\label{l4}
		\int_{X_4\cap B_k}\prod_{j=1,2}\Big|\sum_{T\in\ZT_{\be, k}[\tau_j]}Ef_{T}\Big|^2&\lessapprox r^{O(\e^2)}\mu (m r)^{-1}\|f\|_2^2\sup_\theta\|f_\theta\|_{L^2_{avg(\theta)}}^2\\  \nonumber
		&\lessapprox r^{O(\e^{2})}(\la r)^{-3/4}\|f\|_2^2\;\sup_\theta\|f_\theta\|_{L^2_{avg(\theta)}}^2.
	\end{align}
    This gives us a first estimate.

	Since $|T\cap ( X\cap B_k)|\lesssim \la  |T|$ for all $T\in\ZT_{\beta,k}$, by Cauchy-Schwarz and by Lemma~\ref{lem: l2} we get a second estimate  
	\begin{equation}
		\label{l2-1}
		\int_{X_4\cap B_k}\prod_{j=1,2}\Big|\sum_{T\in\ZT_{\be, k}[\tau_j]}Ef_{T}\Big|\lesssim (\la  r)\|f\|_2^2.
	\end{equation}
	Therefore, since $K=r^{O(\e^2)}$, by \eqref{X-4}, $\eqref{l4}^{4/7}\cdot\eqref{l2-1}^{3/7}$ gives when $p=22/7$,
	\begin{align}
		\nonumber
		\int_{B_r}|\Br_{A} Ef'|^p\lessapprox r^{O(\e^2)} \|f\|_2^2\;\sup_\theta\|f_\theta\|_{L^2_{avg(\theta)}}^{p-2}\leq C_\e R^{\e}r^\e\|f\|_2^2\;\sup_\theta\|f_\theta\|_{L^2_{avg(\theta)}}^{p-2}.
	\end{align}
	This proves \eqref{mixed-norm-2} and hence Proposition \ref{broad-prop}. \qedhere

\end{proof}

\medskip 

\subsection{Proof of Theorem \ref{restriction-thm-local}}
Finally, let us prove Theorem \ref{restriction-thm-local} using Proposition \ref{broad-prop} and a standard induction on scales.

\begin{proof}[Proof of Theorem \ref{restriction-thm-local}]
	Clearly, we may assume $\e$ is small enough.
	We will prove Theorem \ref{restriction-thm-local} by induction on $R$.
	Let $\vp$ be a Schwartz function on $\ZR^2$ that equals to $1$ on $B^2(0,R^2)$ and decays rapidly outside the ball.
	Take $g=f\ast\wh\vp$.
	Then $|Ef(x)-Eg(x)|\leq R^{-1000} \|f\|_2$ when $x\in B_R$.
	Thus, 
	\begin{equation}
		\nonumber
		\Big|\int_{B_R}|Ef|^p-\int_{B_R}|Eg|^p\Big|\lesssim R^{-1000}\|f\|_2^p.
	\end{equation}
	Note that $\|g\|_\infty \leq R^{10}\|f\|_2$.
	For a dyadic number $\mu\in[R^{-10}, R^{10}]$, let $E_\mu$ be the level set $\{|g|\sim\mu\|g\|_2\}$ and let $E_l=\{|g|\leq R^{-10}\|g\|_2\}$ be the lower level set.
	Define $g_\mu=g\Id_{E_\mu}$ and $g_l=g\Id_{E_l}$.
	Thus,
	\begin{equation}
		\nonumber
		\int_{B_R}|Eg|^p\leq\int_{B_R}|Eg_l|^p+\sum_\mu\int_{B_R}|Eg_\mu|^p.
	\end{equation}
	
	If $\int_{B_R}|Eg|^p\lesssim\int_{B_R}|Eg_l|^p$, then \eqref{restriction-esti-local} is true by the trivial estimate $\int_{B_R}|Eg_l|^p\lesssim R^3\|g_l\|_\infty^p\lesssim R^{-20}\|g\|_2^p\lesssim R^{-20}\|f\|_p^p$.
	Otherwise, by pigeonholing, there exists a $\mu$ such that, by relabeling $h=g_\mu$, we have
	\begin{equation}
		\nonumber
		\int_{B_R}|Eg|^p\lessapprox\int_{B_R}|Eg_\mu|^p=\int_{B_R}|Eh|^p.
	\end{equation}

	\smallskip

	Take $K=R^{\e^{20}}$.
	By \eqref{broad-narrow},
	\begin{align}
		\label{broad-narrow-application}
		\int_{B_R}|Eh|^p\leq K^{5\e}\sum_\tau\int|Eh_{\tau}|+ K^{2\e} \sum_{S_j}\int_{B_R}|Eh_{S_j}|^p+K^{3}\int_{B_R}|\Br_{K^\e} Eh|^p.
	\end{align}
	
	\medskip
	
	Suppose the third term of \eqref{broad-narrow} dominates $\int_{B_R}|Eh|^p$.
	Applying Proposition \ref{broad-prop} with $\e$ replaced by  $\e^2$ and noting that $\|h\|_2^2\|h\|_\infty^{p-2}\sim\|h\|_p^p\leq \|g\|_p^p$, we have
	\begin{align}
		\nonumber
		\int_{B_R}|Eh|^p&\leq K^3 C_\e R^{2\e^2}\|h\|_2^2\sup_{\theta:R^{-1/2}\text{-square}}\|h_\theta\|_{L^2_{avg}(\theta)}^{p-2}\\ \nonumber
		&\lesssim  K^3 C_\e R^{2\e^2}\|h\|_2^2\cdot\|h\|_\infty^{p-2}\lesssim C_\e R^{\e}(K^6R^{2\e^2-\e})\|g\|_p^p.
	\end{align}
	This concludes \eqref{restriction-esti-local} since $K=R^{\e^{20}}$ and since $\|g\|_p\lesssim\|f\|_p$.
	
	\medskip
	
	Suppose the second term in \eqref{broad-narrow-application} dominates.
	Consider each $S_j$ and $\int_{B_R}|Eh_{S_j}|^p$.
	By a suitable affine transformation, we may assume $S_j$ is contained in the horizontal strip $S=\{(\xi_1,\xi_2):|\xi_1|\leq K^{-1}, |\xi_2|\sim1\}$.
	Thus,
	\begin{equation}
		\nonumber
		\int_{B_R}|Eh_{S_j}|^p=\int_{B_R}\Big|\int e^{i(x_1\xi_1+x_2\xi_2+x_3\xi_1\xi_2)}h(\xi)\Id_{S}d\xi_1d\xi_2\Big|^pdx_1dx_2dx_3.
	\end{equation}
	Let $\bar h(\xi_1,\xi_2)=h(K\xi_1,\xi_2)$ and $\Box=\{(x_1,x_2,x_3):|x_1|, |x_3|\leq RK^{-1}, |x_2|\leq R\}$.
	Via the change of variables $\xi_1\to K^{-1}\xi_1$ and $x_1\to Kx_1$, $x_3\to Kx_3$, we have
	\begin{equation}
		\label{rescaling}
		\int_{B_R}|Eh_{S_j}|^p\leq K^{2-p}\int_{\Box}|E\bar h|^p.
	\end{equation}
	Partition $\Box$ into finite-overlapping $RK^{-1}$-balls $\{B_j\}$.
	For each $B_j$, let $\bar h_j$ be the sum of scale $RK^{-1}$ wave packets associated with tubes intersecting $B_j$, so that
	\begin{equation}
		\nonumber
		\int_{B_j}|E\bar h|^p\lesssim\int_{B_j}|E\bar h_j|^p+R^{-1000}\|h\|_2^p.
	\end{equation}
	Apply Theorem \ref{restriction-thm-local} at the smaller scale $RK^{-1}$ so that
	\begin{equation}
		\label{rest-ind-hypo-1}
		\int_{B_j}|E\bar h_j|^p\leq C_\e R^\e K^{-\e}\|\bar h_j\|_p^p.
	\end{equation}
	Since the  Fourier transforms of $\{\bar h_j\}$ are contained in finite-overlapping $RK^{-1}$-balls in $\ZR^2$, we have $\sum_j\|\bar h_j\|_p^p\lesssim \|\bar h\|_p^p$.
	Thus, we can sum up over all $j$ in \eqref{rest-ind-hypo-1} to get
	\begin{equation}
		\nonumber
		\int_{\Box}|E\bar h|^p\lesssim C_\e R^\e K^{-\e}\|\bar h\|_p^p= C_\e R^{\e}K^{1-\e}\| h_{S_j}\|_p^p.
	\end{equation}
	Note that $\{S_j\}$ are finite-overlapping and recall \eqref{rescaling}. 
	Put this back to \eqref{broad-narrow-application} and sum up the contributions from all $S_j$ to get
	\begin{equation}
	\nonumber
		\int_{B_R}|Eh|^p\lesssim C_\e R^{\e} K^{3-p+\e}\sum_j\|h_{S_j}\|_p^p\lesssim C_\e R^{\e} K^{3-p+\e}\|h\|_p^p.
	\end{equation}
	This concludes \eqref{restriction-esti-local} as $p>3$ and $K=R^{\e^{20}}\gg 1$.
	
	\smallskip
	
	The proof for the case when the first term in \eqref{broad-narrow-application} dominates is similar to the case when the second term dominates, and we leave the details to the reader. \qedhere

\end{proof}

\bigskip

\bibliographystyle{alpha}
\bibliography{bibli}

\end{document}